\newif\ifarXiv         % declaration
\newif\ifjournal        % declaration
 \journalfalse       \arXivtrue      % set journal to false  and arXiv true when we are using arXiv form
\ifjournal
\documentclass[review,onefignum,onetabnum]{siamart220329}
\usepackage{lipsum}
\usepackage{amsfonts}
\usepackage{amsmath,booktabs,ctable,threeparttable,boxedminipage}
\usepackage{amssymb,amsfonts,boxedminipage}
\usepackage{graphicx,epstopdf} % <- Preamble

\usepackage{subfigure}

\usepackage{algorithm}
\usepackage{algpseudocode}
\usepackage{blkarray}
\usepackage{arydshln}
\usepackage{lineno}
\usepackage{xcolor}
\usepackage{cite}

\ifpdf
  \DeclareGraphicsExtensions{.eps,.pdf,.png,.jpg}
\else
  \DeclareGraphicsExtensions{.eps}
\fi
\newsiamremark{remark}{Remark}
\newsiamremark{hypothesis}{Hypothesis}
\crefname{hypothesis}{Hypothesis}{Hypotheses}
\newsiamthm{claim}{Claim}

\usepackage{amsopn}

\headers{iDARR: iterative data-adaptive RKHS regularization}{H. Li, J. Feng, and F. Lu}

\title{{Scalable iterative data-adaptive RKHS regularization}\thanks{Submitted to the editors DATE.
\funding{}}}

\author{Haibo Li\thanks{School of Mathematics and Statistics, The University of Melbourne, Victoria, Australia.
  (\email{haibo.li@unimelb.edu.au}).}
\and Jinchao Feng \thanks{Department of Mathematics, School of Sciences, Great Bay University, Dongguan, China.
  (\email{jcfeng@gbu.edu.cn}).}
\and Fei Lu
\thanks{Department of Mathematics, Johns Hopkins University, Baltimore, USA.
  (\email{feilu@math.jhu.edu}).}
}

\def\vecphi{x}  %% previouly using \phi 
\def\obs{b}       %% previous, y

\def\bf{\mathbf{f}} 
 \def\bw{\mathbf{w}} 

\def\bA{\mathbf{A}}   

\def\bb{\mathbf{b}}

\def\R{\mathbb{R}}

\def\mH{\mathcal{H}}

\newcommand{\norm}[1]{\left\|#1\right\|}

\newcommand{\innerp}[1]{\langle{#1}\rangle}

\newcommand{\mathspan}[1]{ \mathrm{span}\left\{ {#1} \right\} }

\newcommand\argmin{\mathop{\mathrm{argmin}}}
\def\calE{\mathcal{E}}
\def\mathspan{\mathrm{span}}
\def\Gbar{ {\overline{G}} }

\def\LGbar{ {\mathcal{L}_{\overline{G}}}  }

\def\mH{\mathcal{H}}

\def\calS{\mathcal{S}}
\def\calT{\mathcal{T}}

\def\spaceY{{L^2_\mu(\mathcal{T})}} 
\newcommand{\beqa}{\begin{eqnarray}}
\newcommand{\eeqa}{\end{eqnarray}}
\newcommand{\beqas}{\begin{eqnarray*}}
\newcommand{\eeqas}{\end{eqnarray*}}
\newcommand{\beq}{\begin{equation}}
\newcommand{\eeq}{\end{equation}}
\newcommand{\beqs}{\begin{equation*}}
\newcommand{\eeqs}{\end{equation*}}

\numberwithin{equation}{section}
\numberwithin{theorem}{section}

\usepackage{xcolor}
\definecolor{darkmagenta}{rgb}{0.55, 0.0, 0.55}
\colorlet{colorYO}{darkmagenta}
\renewcommand{\algorithmicrequire}{\textbf{Input:}}    
\renewcommand{\algorithmicensure}{\textbf{Output:}}

\fi

\ifarXiv     % arXiv
\documentclass[11pt]{article}
\usepackage{amsmath,amssymb,amsfonts,mathabx,setspace}
\usepackage{graphicx,caption,epsfig,subfigure,epsfig,wrapfig,float}
\usepackage{url,color,verbatim,algorithmicx}
\usepackage[noend]{algpseudocode}
 \usepackage[ruled,boxed]{algorithm} % {algorithm}% [ruled]{algorithm}  % algorithm;

\usepackage{enumitem}
\usepackage{booktabs}  % for \toprule
\usepackage[sort,compress]{cite}
\usepackage[scaled]{helvet}
\usepackage[T1]{fontenc}
\usepackage[bookmarks=true, bookmarksnumbered=true, colorlinks=true,   pdfstartview=FitV,
linkcolor=blue, citecolor=blue, urlcolor=blue]{hyperref}
 \usepackage{authblk}
\usepackage{arydshln}
\usepackage[sort,compress]{cite}
\let\oldbibliography\thebibliography
\renewcommand{\thebibliography}[1]{%
  \oldbibliography{#1}%
  \setlength{\itemsep}{-2pt}%
}

\usepackage[normalem]{ulem}%for \sout (strikeout)
\usepackage{mathtools}% for \mathsout
\newcommand{\mathsout}[1]% will draw line through middle of #1
{\bgroup\mathchoice
  {\sbox0{$\displaystyle{#1}$}%
    \usebox0\hspace{-\wd0}%
    \rule[0.5\ht0-0.5\dp0-.5pt]{\wd0}{1pt}}%
  {\sbox0{$\textstyle{#1}$}%
    \usebox0\hspace{-\wd0}%
    \rule[0.5\ht0-0.5\dp0-.5pt]{\wd0}{1pt}}%
  {\sbox0{$\scriptstyle{#1}$}%
    \usebox0\hspace{-\wd0}%
    \rule[0.5\ht0-0.5\dp0-.5pt]{\wd0}{1pt}}%
  {\sbox0{$\scriptscriptstyle{#1}$}%
    \usebox0\hspace{-\wd0}%
    \rule[0.5\ht0-0.5\dp0-.5pt]{\wd0}{1pt}}%
\egroup}
\def\vecphi{x}  %% previouly using \phi 
\def\obs{b}       %% previous, y

\def\bf{\mathbf{f}} 
 \def\bw{\mathbf{w}} 

\def\bA{\mathbf{A}}   

\def\bb{\mathbf{b}}

\def\R{\mathbb{R}}

\def\mH{\mathcal{H}}

\newcommand{\norm}[1]{\left\|#1\right\|}

\newcommand{\innerp}[1]{\langle{#1}\rangle}

\newcommand{\mathspan}[1]{ \mathrm{span}\left\{ {#1} \right\} }

\newcommand\argmin{\mathop{\mathrm{argmin}}}
\def\calE{\mathcal{E}}
\def\mathspan{\mathrm{span}}
\def\Gbar{ {\overline{G}} }

\def\LGbar{ {\mathcal{L}_{\overline{G}}}  }

\def\mH{\mathcal{H}}

\def\calS{\mathcal{S}}
\def\calT{\mathcal{T}}

\def\spaceY{{L^2_\mu(\mathcal{T})}} 
\newcommand{\beqa}{\begin{eqnarray}}
\newcommand{\eeqa}{\end{eqnarray}}
\newcommand{\beqas}{\begin{eqnarray*}}
\newcommand{\eeqas}{\end{eqnarray*}}
\newcommand{\beq}{\begin{equation}}
\newcommand{\eeq}{\end{equation}}
\newcommand{\beqs}{\begin{equation*}}
\newcommand{\eeqs}{\end{equation*}}

\newtheorem{theorem}{Theorem}

\newtheorem{definition}[theorem]{Definition}

\newtheorem{lemma}[theorem]{Lemma}
\newtheorem{proposition}[theorem]{Proposition}

\newenvironment{proof}[1][Proof]{\noindent\textbf{#1.} }{\ \rule{0.5em}{0.5em}}

\numberwithin{equation}{section}
\numberwithin{theorem}{section}

\usepackage{xcolor}
\definecolor{darkmagenta}{rgb}{0.55, 0.0, 0.55}
\colorlet{colorYO}{darkmagenta}
\title{Scalable iterative data-adaptive RKHS regularization}
\date{}
\author[1]{Haibo Li}
\author[2]{Jinchao Feng} 
\author[3] {Fei Lu}
\affil[1]{\small School of Mathematics and Statistics, The University of Melbourne, Victoria, Australia. }% (\email{haibo.li@unimelb.edu.au}}
\affil[2]{\small  Department of Mathematics, School of Sciences, Great Bay University, Dongguan, China.}%  (\email{jcfeng@gbu.edu.cn})}
\affil[3]{\small  Department of Mathematics, Johns Hopkins University, Baltimore, USA.}% (\email{feilu@math.jhu.edu})}

\usepackage{amsopn}

\usepackage{cleveref}
\fi

\ifpdf
\hypersetup{
  pdftitle={Scalable iterative data-adaptive RKHS regularization},
  pdfauthor={H. Li, J. Feng and F. Lu}
}
\fi

\begin{document}
\ifjournal    \pagewiselinenumbers    \fi
\maketitle

\ifarXiv \vspace{-12mm} \fi
% REQUIRED
\begin{abstract}
We present iDARR, a scalable iterative Data-Adaptive RKHS Regularization method, for solving ill-posed linear inverse problems. The method searches for solutions in subspaces where the true solution can be identified, with the data-adaptive RKHS penalizing the spaces of small singular values. At the core of the method is a new generalized Golub-Kahan bidiagonalization procedure that recursively constructs orthonormal bases for a sequence of RKHS-restricted Krylov subspaces. The method is scalable with a complexity of $O(kmn)$ for $m$-by-$n$ matrices with $k$ denoting the iteration numbers. Numerical tests on the Fredholm integral equation and 2D image deblurring show that it outperforms the widely used $L^2$ and $l^2$ norms, producing stable accurate solutions consistently converging when the noise level decays.
\end{abstract}

\ifjournal
% REQUIRED

% REQUIRED
% 	47A52 Linear operators and ill-posed problems, regularization
%	65F22  	Ill-posedness and regularization problems in numerical linear algebra
% 	65J20  	Numerical solutions of ill-posed problems in abstract spaces; regularization
\fi

\ifarXiv

\vspace{-4mm}
% \tableofcontents
\fi

%%%%%%%%%%%%%%%%%%%%%%%%%%%%%%
%%%%%%%%  ====== section =======  %%%%%%%%
%%%%%%%%%%%%%%%%%%%%%%%%%%%%%%
\section{Introduction}
% 1. Ill-posedness ->>> need regularization 
% >>> brief review setting/method >>> this study: norm and iterative. 
% 2. Main results, and contributions 
This study considers large-scale ill-posed linear inverse problems with little prior information on the regularization norm. The goal is to reliably solve high-dimensional vectors $x\in \R^n$ from the equation
\begin{equation}\label{eq:Ab}
	A x + \bw = b , \quad A\in \R^{m\times n}, 
\end{equation}
where $A$ and $b$ are data-dependent forward mapping and output, and $\bw$ denotes noise or measurement error.  
The problem is ill-posed in the sense that the least squares solution with minimal Euclidean norm, often solved by $x_{LS}=A^{\dagger}b$ or $x_{LS}= (A^\top A)^\dagger A^\top b$ with $^\dagger$ denoting pseudo-inverse, is sensitive to perturbations in $b$. Such an ill-posedness happens when the singular values of $A$ decay to zero 
%to and cluster at zero, and the positive singular values of $A$ decay 
faster than the perturbation in $b$ projected in the corresponding singular vectors. 

Regularization is crucial to producing stable solutions for the ill-posed inverse problem. Broadly, it encompasses two integral components: a penalty term that defines the search domain and a hyperparameter that controls the strength of regularization. There are two primary approaches to implementing regularization: direct methods, which rely on matrix decomposition, e.g., the Tikhonov regularization \cite{tihonov1963solution}, the truncated singular value decomposition (SVD) \cite{hansen1998rank,engl1996regularization}; and iterative methods, which use matrix-vector computations to scale for high-dimensional problems, see e.g., \cite{gazzola2019ir,chen2022stochastic,zhang2022stochastic} for recent developments.

In our setting, we encounter two primary challenges: selecting an adaptive regularization norm and devising an iterative method to ensure scalability. The need for an adaptive norm arises from the variability of the forward map $A$ across different applications and the often limited prior information about the regularity of $x$. Many existing regularization norms, such as the Euclidean norms used in Tikhonov methods in  \cite{hansen1998rank,tihonov1963solution} and the total variation norm in \cite{rudin1992nonlinear}, lack this adaptability; for more examples, see the related work section below. Although a data-adaptive regularization norm has been proposed in \cite{LLA22,LAY22} for nonparametric regression, it is implicitly defined and requires a spectral decomposition of the normal operator.

We introduce iDARR, an \textsf{i}terative \textsf{D}ata-\textsf{A}daptive \textsf{R}eproducing kernel Hilbert space \textsf{R}egularization method. This method resolves both challenges by iteratively solving the subspace-projected problem
\begin{equation*}
	x_k =\argmin_{\vecphi\in\mathcal{X}_k}\|x\|_{ C_{rkhs}}, \ \ \mathcal{X}_k = \{\vecphi: \min_{x\in\mathcal{S}_k}\|Ax-\obs\|_{2}\}, 
\end{equation*}
where $\|\cdot\|_{C_{rkhs}}$ is the implicitly defined semi-norm of a \emph{data-adaptive RKHS} (DA-RKHS), and $\mathcal{S}_k$ are subspaces of the DA-RKHS constructed by a \emph{generalized Golub-Kahan bidiagonalization} (gGKB). By stopping the iteration early using the L-curve criterion, it produces a stable accurate solution without using any matrix decomposition.

The DA-RKHS is a space defined by the data and model, embodying the intrinsic nature of the inverse problem. Its closure is the data-dependent space in which the true solution can be recovered, particularly when $A$ is deficient-ranked. Thus, when used for regularization, it confines the solution search in the right space and penalizes the small singular values, leading to stable solutions. We construct this DA-RKHS by reformulating \cref{eq:Ab} as a weighted Fredholm integral equation of the first kind and examining the identifiability of the input signal, as detailed in \Cref{sec:FSOI}.

Our key innovation is the gGKB. It constructs solution subspaces in the DA-RKHS without explicitly computing it. It is scalable with a cost of only $O(kmn)$, where $k$ is the number of iterations. This cost is orders of magnitude much smaller than the cost of direct methods based on spectral decomposition of $A^\top A$, typically $O(n^3+mn^2)$ operations.

The iDARR and gGKB have solid mathematical foundations. We prove that each subspace $\mathcal{S}_k$ is restricted in the DA-RKHS, thereby named the \emph{RKHS-restricted Krylov subspaces}. It is spanned by the orthonormal vectors produced by gGKB. Importantly, if not stopped early, the gGKB terminates when the RKHS-restricted Krylov subspace is fully explored, and the solution in each iteration is unique.

Systematic numerical tests employing the Fredholm integral equations demonstrate that iDARR surpasses traditional iterative methods employing $l^2$ and $L^2$ norms in the state-of-the-art \textsf{IR TOOLS} package \cite{gazzola2019ir}. Notably, iDARR delivers accurate estimators that consistently decay with the noise level. This superior performance is evident irrespective of whether the spectral decay is exponential or polynomial, or whether the true solution resides inside or outside the DA-RKHS. Furthermore, our application to image deblurring underscores both its scalability and accuracy.

% Our major contribution is the innovative adaptation of the GKB to construct solution subspaces, uniquely tailored through a data-adaptive RKHS norm. This innovation is underscored by our use of the explicit matrix-form expression of $ C_{rkhs}^{\dag}$, a strategy that significantly enhances computational efficiency.  In contrast to conventional regularization methods, which often wrestle with the computational burdens of $C^{-1}$ when using the norm $\|x\|_C$ --- notably when $C= L^\top L$ for the norm $\|Lx\|_2$ -- our approach distinctively bypasses these complexities and related numerical errors. Thus, our gGKB method stands out in efficiency and effectiveness.  

\noindent\textbf{Main contributions.}
Our main contribution lies in developing iDARR, a scalable iterative regularization method tailored for large-scale ill-posed inverse problems with little prior knowledge about the solution. The cornerstone of iDARR is the introduction of a new data-adaptive RKHS determined by the underlying model and the data. A key technical innovation is the gGKB, which efficiently constructs solution subspaces of the implicitly defined DA-RKHS.  

\subsection{Related work}
Numerous regularization methods have been developed, and the literature on this topic is too vast to be surveyed here; we refer to \cite{hansen1998rank,engl1996regularization,gazzola2019ir} and references therein for an overview. In the following, we compare iDARR with the most closely related works.  

\paragraph{Regularization norms}
Various regularization norms exist, such as Euclidean norms of Tikhonov in \cite{hansen1998rank,tihonov1963solution}, the total variation norm $\|\vecphi'\|_{L^1}$ of the  Rudin--Osher--Fatemi method in \cite{rudin1992nonlinear}, the $L^1$ norm $ \|\vecphi\|_{L^1}$ of LASSO in \cite{tibshirani1996_RegressionShrinkage}, and the RKHS norm $\|\vecphi\|_R^2$ of an RKHS with a user-specified reproducing kernel $R$ \cite{wahba1977practical,bauer2007regularization,CZ07book}. These norms, however, are often based on presumed properties of the solution and do not consider the specifics of each inverse problem. Our RKHS norm differs by adapting to the model and data: our RKHS has a reproducing kernel determined by the inverse problem, and its closure is the space in which the solution can be identified, making it an apt choice for regularization in the absence of additional solution information.

\paragraph{Iterative regularization (IR) methods} 
IR methods are scalable by accessing the matrix only via matrix-vector multiplications, producing a sequence of estimators until an early stopping, where the iteration number plays the role of the regularization parameter. IR has a rich and extensive history and continues to be a vibrant area of interest in contemporary studies \cite{Paige1982,Arridge2014,gazzola2019ir,Li2023}. Different regularization terms lead to various methods. The LSQR algorithm \cite{Paige1982,Bjorck1988} with early stopping is standard for $\|x\|_{2}^2$-regularization. It solves projected problems in Krylov subspaces before transforming back to the original domain. For $\|Lx\|_{2}^{2}$ with $L\in\mathbb{R}^{p\times n}$, the widely-used methods include joint bidiagonalization method \cite{Kilmer2007Projection,JiaYang2020}, generalized Krylov subspace method \cite{Lampe2012large,Reichel2012tikhonov}, random SVD or generalized SVD method \cite{xiang2013regularization,Xiang2015,Wei2016}, modified truncated SVD method \cite{Bai2021novel,Huang2022tikhonov}, etc. For the general regularization norm in the form $x^TMx$ with a symmetric matrix $M$, the MLSQR in \cite{Arridge2014} treats positive definite $M$ and the preconditioned GKB \cite{Li2023} handles positive semi-definite $M$'s.  Our iDARR studies the case that $M$ is unavailable but $M^{\dagger}$ is ready to be used.

\paragraph{Golub-Kahan bidiagonalization (GKB)} The GKB was first used to solve inverse problems in \cite{Oleary1981}, which generates orthonormal bases for Krylov subspaces in $(\mathbb{R}^{n},\langle\cdot,\cdot\rangle_2)$ and $(\mathbb{R}^{m},\langle\cdot,\cdot\rangle_2)$. This method extends to bounded linear compact operators between Hilbert spaces, with properties and convergence results detailed in \cite{caruso2019convergence}. Our gGKB extends the method to construct RKHS-restricted Krylov subspace in $(\mathbb{R}^{n},\langle\cdot,\cdot\rangle_{C_{rkhs}})$ and $(\mathbb{R}^{m},\langle\cdot,\cdot\rangle_2)$, where $C_{rkhs}$ is positive semidefinite, and in particular, only $C_{rkhs}^\dag$ is available.   

% \bigskip 
The remainder of this paper is organized as follows: \Cref{sec:FSOI} introduces the adaptive RKHS with a characterization of its norm. \Cref{sec:alg-idartr} presents in detail the iDARR. \Cref{sec:gGKB} proves the desired properties of gGKB. In \Cref{sec:num}, we systematically examine the algorithm and demonstrate the robust convergence of the estimator when the noise becomes small. Finally, \Cref{sec:conlusion} concludes with a discussion on future developments.

%%%%%%%%%=================%%%%%%%%%%%
%%%%%%%%%=================%%%%%%%%%%%
%%%%%%%%%=================%%%%%%%%%%%
\section{A Data Adaptive RKHS for Regularization} \label{sec:FSOI} 
This section introduces a data-adaptive RKHS (DA-RKHS) that adapts to the model and data in the inverse problem. The closure of this DA-RKHS is the function (or vector) space in which the true solution can be recovered, or equivalently, the inverse problem is well-defined in the sense that the loss function has a unique minimizer. Hence, when its norm is used for regularization, this DA-RKHS ensures that the minimization searches in the space where we can identify the solution.

To describe the DA-RKHS, we first present a unified notation that applies to both discrete and continuous time models using a weighted Fredholm integral equation of the first kind. Based on this notation, we write the normal operator as an integral operator emerging in a variational formulation of the inverse problem. The integral kernel is the reproducing kernel for the DA-RKHS. In other words, the normal operator defines the DA-RKHS. At last, we briefly review a DA-RKHS Tikhonov regularization algorithm, the DARTR algorithm.

\subsection{Unified notation for discrete and continuous models}\label{sec:unif-notation}

The linear equation \cref{eq:Ab} can arise from discrete or continuous inverse problems. In either case, we can present the inverse problem using the prototype of the Fredholm integral equation of the first kind. We consider only the 1D case for simplicity, and the extension to high-dimensional cases is straightforward.  Specifically, let $\calS,\calT\subset \R$ be two compact sets. We aim to recover the function $\phi:\calS\to \R$ in the Fredholm equation%  \cref{eq:FIE}
\begin{align}\label{eq:FIE}
 y(t) = \int_\calS K(t,s)\phi(s) \nu(ds)+ \sigma \dot W(t) =: L_K\phi(t) + \sigma \dot W(t), \quad \forall t\in \calT
 \end{align}
from discrete noisy data 
$$
b = (y(t_1),\ldots, y(t_m))^\top\in \R^m,
$$ where we assume the observation index $\calT= \{t_j\}_{j=1}^m$ to be $0=t_0 < t_1 <\cdots <t_m$ for simplicity. Here the measurement noise  $\sigma \dot W(t)$ is the white noise scaled by $\sigma$; that is, the noise at $t_j$ has a Gaussian distribution $\mathcal{N}(0, \sigma^2 (t_{j+1}-t_j))$ for each $j$. Such noise is integrable when the observation mesh refines, i.e., $\max_j(t_{j+1}-t_j)$ vanishes. 
% Correspondingly, we define $\mu$ as an atomic measure with $\mu(t_i) =t_{i+1}-t_i$, which can be viewed as the discrete approximation of the Lebesgue measure. 

Here the finite measure $\nu$ can be either the Lebesgue measure with $\calS$ being an interval or an atom measure with $\calS$ having finitely many elements. Correspondingly, the Fredholm integral equation \cref{eq:FIE} is either a continuous or a discrete model. 

In either case, the goal is to solve for the function $\phi:\calS\to \R$ in \cref{eq:FIE}. When seeking a solution in the form of $\phi = \sum_{i=1}^n x_i \phi_i$, where $\{\phi_i\}_{i=1}^n$ is a pre-selected set of basis functions, we obtain the linear equation \cref{eq:Ab} with $x= (x_1,\ldots,x_n)^\top \in \R^n$ and the matrix $A$ with entries 
\begin{equation}\label{eq:A_basis}
  A(j,i) =\int_\calS K(t_j,s)\phi_i(s)\nu(ds) = L_K\phi_i(t_j), \quad 1\leq j\leq m, 1\leq i\leq n. 	
\end{equation}
In particular, when $\{\phi_i\}$ are piece-wise constants, we obtain $A$ as follows. 
\begin{itemize}
\item \emph{Discrete model.}  Let $\nu$ be an atom measure on $\calS= \{s_i\}_{i=1}^n$, a set with $n$ elements. Suppose that the basis functions are $\phi_i(s) = \mathbf{1}_{\{s_i\}}(s)$. Then, $\phi= x$ and the matrix $A$ has entries $A(j,i) = K(t_j,s_i)\nu(s_i)$. 
 \item \emph{Continuous model.} Let $\nu$ be the Lebesgue measure on $\calS = [0,1]$, and $\phi_i(s)=\mathbf{1}_{[s_{i-1},s_{i}]}(s) $ be piecewise constant functions on a partition of $\calS $ with $0= s_0<s_1<\ldots < s_n= 1$. Then, $\phi= \sum_{i=1}^n x_i \phi_i$ and the matrix $A$ has entries $A(j,i) = K(t_j,s_i) (s_i-s_{i-1})$.   
\end{itemize}

The default function spaces for $\phi$ and $y$ above are $L^2_\nu(\calS)$ and $L^2_\mu(\calT)$. % The integral operator $L_K:L^2_\nu(\calS)\to L^2_\mu(\calT)$ is compact since $K$ is continuous. 
The loss function $\calE(x) = \|Ax-b\|_{2}^2$ over $L^2_\nu(\calS)$ becomes 
\begin{align}\label{least_sq_org}
\mathcal{E}(\phi):=\norm{L_K\phi-y}^2_{\spaceY}& = \innerp{L_K\phi,L_K\phi}_{\spaceY}-2\innerp{L_K\phi,y}_{\spaceY}+\norm{y}^2_{\spaceY}, 
% & = \innerp{\LGbar\phi,\phi} - 2\innerp{\phi^D,\phi}+ const.
\end{align}

Eq.\cref{eq:FIE} is a prototype of ill-posed inverse problems, dating back from Hadamard \cite{hadamard1923lectures}, and it remains a testbed for new regularization methods \cite{wahba1973convergence,nashed1974generalized,hansen1998rank,Li2005modified}.  
%It is ill-posed in the sense that the least squares solution with mini-norm is sensitive to small measurement noise in data \cite{nashed1974generalized}. % Such an ill-posedness roots in the fact that the operator $L_K$ is compact with positive eigenvalues decaying to zero. 

The $L^2_\nu(\calS)$ norm is often a default choice for regularization. However, it has a major drawback: it does not take into account the operator $L_K$, particularly when $L_K$ has zero eigenvalues, and it leads to unstable solutions that may blow up in the small noise limit \cite{LangLu23sna}. To avoid such instability, particularly for iterative methods, we introduce a weighted function space and an RKHS that are adaptive to both the data and the model in the next sections.

%The basis matrix $B$ of $\{\phi_i\}$ is  $B(i,j) = \innerp{\phi_i,\phi_j}_{L^2_\rho}$.

%%%%%%%%%%%%%%=== 
\subsection{The function space of identifiability} 
We first introduce a weighted function space $L^2_\rho(\calS)$, where the measure $\rho$ is defined as  
\beq
\label{eq:exp_measure}
\frac{d\rho}{d\nu}(s):=\frac{1}{Z} \int_\calT | K(t,s)| \mu(dt), \, \forall s\in \calS, 
\eeq
where $Z= \int_\calS \int_\calT | K(t,s) |\mu(dt)\nu(ds)$ is a normalizing constant. % It is well-defined since $K$ is continuous and $\calT$ is compact.
 This measure quantifies the exploration of data to the unknown function through the integral kernel $K$ at the output set $\calT$, i.e., $\{K(t_j,\cdot)\}_{t_j\in \calT}$, hence it is referred to as an \emph{exploration measure}. In particular, when \eqref{eq:Ab} is a discrete model, the exploration measure is the normalized column sum of the absolute values of the matrix $A$.

The major advantage of the space $L^2_\rho(\calS)$ over the original space $L^2_\nu(\calS)$ is that it is adaptive to the specific setting of the inverse problem. In particular, this weighted space takes into account the structure of the integral kernel and the data points in $\calT$.  Thus, while the following introduction of RKHS can be carried out in both $L^2_\rho(\calS)$ and $L^2_\nu(\calS)$, we will focus only on $L^2_\rho(\calS)$.

Next, we consider the variational inverse problem over $L^2_\rho$, and the goal is to find a minimizer of the loss function \cref{least_sq_org} in $L^2_\rho$. Since the loss function is quadratic, its Hessian is a symmetric positive linear operator, and it has a unique minimizer in the linear subspace where the Hessian is strictly positive. We assign a name to this subspace in the next definition. 
% \begin{definition}[normal operator] In a variational inverse problem of minimizing a quadratic loss function $\calE$ in a Hilbert space $L^2_\rho$, we call the Hessian of $\calE$ as the normal operator, and denote it by $\LGbar$. 
% \end{definition}
\begin{definition}[Function space of identifiability] 
In a variational inverse problem of minimizing a quadratic loss function $\calE$ in $L^2_\rho$, we call $\LGbar= \frac{1}{2}\nabla^2\calE$ the \emph{normal operator}, where $\nabla^2\calE$ is the Hessian of $\calE$, and we call 
 $
 H= \mathrm{Null}(\LGbar)^\perp 
 % :=\{\phi\in L^2_\rho: \LGbar\phi \neq 0 \}\bigcup\{0\}.  
 $
  the \emph{function space of identifiability (FSOI)}.
 \end{definition}
  
The next lemma specifies the FSOI for the loss function in \cref{least_sq_org} (see \cite{LLA22} for its proof).  
%%%%
\begin{lemma}\label{thm:FSOI} 
Assume $K\in C(\calT\times \calS)$. For $\rho$ in \cref{eq:exp_measure}, define $\Gbar:\calS\times \calS \to \R$ as  
 \begin{equation}  \label{eq:G}
 \begin{aligned}
 \Gbar(s,s')&:=% \one_{\calS}(s)  \one_{\calS}(s')
 \frac{G(s,s')}{\frac{d\rho}{d\nu}(s)\frac{d\rho}{d\nu}(s')}, \quad   G(s,s')&:= \int_\calT K(t,s)K(t,s') \mu(dt).  	 	
 \end{aligned}
 \end{equation}
\begin{itemize}
\item[{\rm(a)}] The normal operator for $\calE$ in \cref{least_sq_org} over $L^2_\rho$ is $\LGbar: L^2_\rho \to L^2_\rho$ defined by 
\beq
\label{eq:LG}
\LGbar\phi(s):=\int_{\calS} \phi(s')\Gbar(s,s')\rho(ds'), 
\eeq
and the loss function can be written as  
\begin{align}\label{eq:lossFn-sq}
\mathcal{E}(\phi)%:=\norm{L_K\phi-y}^2_{\spaceY}& = \innerp{L_K\phi,L_K\phi}_{\spaceY}-2\innerp{L_K\phi,y}_{\spaceY}+\norm{y}^2_{\spaceY}, 
& = \innerp{\LGbar\phi,\phi}_{L^2_\rho} - 2\innerp{\phi^D,\phi}_{L^2_\rho}+ const.,
\end{align}
where $\phi^D$ comes from Riesz representation s.t.~$\innerp{\phi^D,\phi}_{L^2_\rho}  =\innerp{L_K\phi,y}_{\spaceY} $ for any $\phi\in L^2_\rho$.
\item[\rm{(b)}] 
$\LGbar$ is compact, self-adjoint, and positive. Hence, its eigenvalues $\{\lambda_i\}_{i\geq 1}$ converge to zero and its orthonormal eigenfunctions $\{ \psi_i\}_{i}$ form a complete basis of $L^2_\rho(\calS)$.  
\item[\rm{(c)}] 
The FSOI is $H:=\overline{ \mathspan\{ \psi_i\}_{i:\lambda_i>0} }\subset L^2_\rho(\calS)$, and the unique minimizer of $\calE$ in H is 
$\widehat{\phi} = \LGbar^{-1}\phi^D$, 
where $\LGbar^{-1}$ is the inversion of $\LGbar:H\to L^2_\rho$.  
%and it is unbounded when $\lambda_i> 0$ for all $i$.  
\end{itemize}
\end{lemma}

\Cref{thm:FSOI} reveals the cause of the ill-posedness, and provides insights on regularization: 
\begin{itemize}
\item The variational inverse problem is well-defined only in the FSOI $H$, which can be a proper subset of $L^2_\rho$. Its ill-posedness in $H$ depends on the smallest eigenvalue of the operator $\LGbar$ and the error in $\phi^D$. 
\item When the data is noiseless, the loss function can only identify the $H$-projection of the true input function. When data is noisy, its minimizer $\LGbar^{-1}\phi^D$ is ill-defined in $L^2_\rho$ when $\phi^D\notin \LGbar(H)$.  
 %$\sum_{i:\lambda_i>0} \lambda_i^{-1}=\infty$, but it is well-defined when $\sum_{i:\lambda_i>0} \lambda_i^{-1}<\infty$.
\end{itemize}
As a result, when regularizing the ill-posed problem, an important task is to ensure the search takes place in the FSOI and to remove the noise-contaminated components making $\phi^D\notin \LGbar(H)$. 
% The data-adaptive RKHS in the next section can fulfill the task when its norm is used for regularization.

\subsection{A Data-adaptive RKHS}\label{sec:RKHS}
Our \emph{data-adaptive RKHS} is the RKHS with $\Gbar$ in \cref{eq:G} as a reproducing kernel. Hence, it is adaptive to the integral kernel $K$ and the data in the model. When its norm is used for regularization, it ensures that the search takes place in the FSOI because its $L^2_\rho$ closure is the FSOI; also, it penalizes the components in $\phi^D$ corresponding to the small singular values.  

The next lemma characterizes this RKHS, and we refer to \cite{LLA22} for its proof. 
\begin{lemma}[Characterization of the adaptive RKHS]\label{lemma:rkhs} 
Assume $K\in C(\calT\times\calS)$. The RKHS $H_G$  with $\Gbar$ as its reproducing kernel satisfies the following properties. 
\begin{enumerate}
\item[{\rm(a)}]  $H_G:=\LGbar^{\frac{1}{2}}(L^2_\rho(\calS))$ has inner product 
% \begin{equation*} % \label{eq:rkhs_innerp}
$\innerp{\phi, \phi}_{H_G}=\innerp{\LGbar^{-\frac{1}{2}}\phi,\LGbar^{-\frac{1}{2}} \phi }_{L^2_\rho(\calS)}. 
$ % \end{equation*}
% The operator $\LGbar$ is self-adjoint in $H_G$. Moreover, we have $\innerp{\phi, y }_{L^2_\rho(\calS)}=\innerp{\LGbar \phi, y }_{H_G}$  for any $\phi\in L^2_\rho(\calS)$ and $ y \in H_G$. 
\item [{\rm(b)}] 
  $\{\sqrt{\lambda_i}  \psi_i\}_{\lambda_i>0}$ is an orthonormal basis of $H_G$, where $\{ (\lambda_i,\psi_i)\}_{i}$ are eigen-pairs of $\LGbar$. 
\item [{\rm(c)}] 
For any $\phi=\sum_{i=1}^\infty c_i  y _i\in H_G$, we have 
\begin{equation*}% \label{norms}
\innerp{L_K\phi,L_K\phi}_\spaceY=\sum_{i=1}^\infty \lambda_i c_i^2,\quad \| \phi\|^2_{L^2_\rho}=\sum_{i=1}^\infty c_i^2, \quad \| \phi\|^2_{H_G}=\sum_{i=1}^\infty \lambda_i^{-1} c_i^2. 
\end{equation*}  
% Moreover, the $H_G$ norm is stronger than the $L^2_\rho$-norm: $\norm{\phi}^2_{H_G}\ge \lambda_1^{-1} \norm{\phi}^2_{L^2_\rho}$.
\item [{\rm(d)}]  $H=\overline{ H_G}$ with inclosure in $L^2_\rho(\calS)$, where $H=\overline{ \mathspan\{ y _i\}_{i:\lambda_i>0} }$ is the FSOI. 
\end{enumerate}
\end{lemma}
% \begin{proof}
% The first part is from the standard characterization theorem of RKHS, e.g., \cite[Section 4.4]{CZ07book} or \cite{LLA22}. When $\phi\in L^2_\rho(\calS)$, we have $\LGbar\phi \in H_G$. Then, by the definition of the inner product and the symmetry of $\LGbar^{-1/2}$, we have $\innerp{\LGbar\phi, y }_{H_G} = \innerp{\LGbar^{1/2}\phi,\LGbar^{-1/2} y }_{L^2_\rho(\calS)} = \innerp{\phi, y }_{L^2_\rho(\calS)}$. 
%
%Part \rm{(b)} follows directly from the characterization of the inner product in Part I. Part (c) follows directly from \cref{quadratic_2}, the orthonormality of the eigenfunctions and the characterization of the inner product. Part \rm{(d)} is obvious because both function spaces have the same basis functions. 
% \end{proof}

The next theorem shows the computation of the RKHS norm for the problem \cref{eq:Ab} when it is written in the form \cref{eq:FIE}-\cref{eq:A_basis}. A key component is solving the eigenvalues of $\LGbar$ through a generalized eigenvalue problem. 
% We refer to \cite{LLA22} for its proof. 

\begin{theorem}[Computation of RKHS norm]\label{thm:RKHS_norm}
Suppose that \cref{eq:Ab} is equivalent to \cref{eq:FIE} under basis functions $\{\phi_i\}_{i=1}^n$ with $n\leq \infty$ and \cref{eq:A_basis}.     
%Assume that $\mathrm{span}\{\phi_i\}_{i=1}^n \supseteq \LGbar(L^2_\rho) $ with $n\leq \infty$. 
Let $B$ with entries $B(i,j) = \innerp{\phi_i,\phi_j}_{L^2_\rho}$ be the non-singular basis matrix, where $\rho$ is the measure defined in \cref{eq:exp_measure}. 
Then, the operator $\LGbar$ in \cref{eq:LG} has eigenvalues $(\lambda_1,\ldots,\lambda_n)$ solved by the generalize eigenvalue problem: 
\begin{equation}\label{eq:AbB1}
A^\top A V=  B V\Lambda, \quad s.t., \ \ V^\top B V = I_n, \quad \Lambda= \mathrm{diag}(\lambda_1,\ldots,\lambda_n), 
\end{equation}
and the eigenfunctions are  $\{ \psi_k = \sum_{j=1}^n  V_{jk}\phi_j\}_k$. The RKHS norm of $\phi=\sum_{i=1}^{n}x_i\phi_i$ satisfies 
\begin{equation}\label{eq:mat_rkhs}
\begin{aligned}
	\| \phi\|_{H_G}^2& =\|x\|_{C_{rkhs}}^2= x^\top  C_{rkhs} x, \\
	 C_{rkhs} & =  (V\Lambda V^\top)^{\dagger} = B (A^\top A)^\dagger B, \quad  C_{rkhs}^{\dagger}= B^{-1}(A^\top A) B^{-1}. 
\end{aligned}
\end{equation}
In particular, if $B=I_n$, we have $C_{rkhs} = (A^\top A)^\dagger$. 
\end{theorem} 

% \iffalse
\begin{proof}
	% [Proof of \Cref{thm:RKHS_norm}] 
Denote $\boldsymbol{\Phi} =(\phi_1,\ldots,\phi_n)^\top$ and $\boldsymbol{\Psi} =(\psi_1,\ldots,\psi_n)^\top$. 
We first prove that the eigenvalues of $\LGbar$ are solved by \cref{eq:AbB1}. We suppose $\{(\lambda_i,\psi_i)\}_{i=1}^n$ are the eigenvalues and eigen-functions of $\LGbar$ over $L^2_\rho$ with $\{\psi_i\}$ being an orthonormal basis of $L^2_\rho(\calS)$. Since $\mH = \mathrm{span}\{\phi_i\}_{i=1}^n \supseteq \LGbar(L^2_\rho) $, there exists $V\in \R^{n\times n}$ such that $\boldsymbol{\Psi}= V^\top \boldsymbol{\Phi}$, i.e., $\psi_k = \sum_{j=1}^n  V_{jk}\phi_j\}_k$ for each $1\leq k\leq n$. Then, the task is to verify that $V$ and $\Lambda = \mathrm{Diag}(\lambda_1,\ldots,\lambda_n)$ satisfy $A^\top A V=  B V\Lambda$ and $V^\top B V = I_n$. 
 
The orthonormality of $\{\psi_i\}$ implies that 
 $$I_n = \big(\innerp{\psi_k,\psi_l}_{L^2_\rho}\big)_{1\leq k,l\leq n} =   \big(\innerp{\sum_{i=1}^nV_{ik}\phi_i,\sum_{j=1}^nV_{jl} \phi_j}_{L^2_\rho}\big)_{1\leq k,l\leq n} =  V^\top B V. $$ 
 Note that $[A^\top A](j,i) = \innerp{L_K\phi_i, L_K\phi_j}_{L^2_\mu(\calT)} = \innerp{\phi_j,\LGbar \phi_i}_{L^2_\rho}
 $ for $1\leq i,j\leq n$. Then, 
 \[ \innerp{\phi_j,\LGbar\psi_k}_{L^2_\rho} = \sum_{i=1}^n [A^\top A](j,i) V_{i k}.
 \]  
 Meanwhile, the eigen-equation $\LGbar \psi_k= \lambda_k\psi_k$ implies that for each $\phi_j$, 
 $$\innerp{\phi_j,\LGbar\psi_k}_{L^2_\rho} =\lambda_k\innerp{\phi_j,\psi_k}_{L^2_\rho} = \lambda_k \innerp{\phi_j ,\sum_{i=1}^n V_{ik}\phi_i}_{L^2_\rho} = \lambda_k \sum_{i=1}^nB_{ji}V_{ik},  
 $$
 i.e., $\big(\innerp{\phi_j,\LGbar\psi_k}_{L^2_\rho}\big) = BV\Lambda$. 
Hence, these two equations imply that $A^\top A V= BV\Lambda$.   

Next, to compute the norm of $\phi=\sum_{i=1}^{n}x_i\phi_i \in H_G$, we write it as $\phi = x^\top \boldsymbol{\Phi}= x^\top V^{-1}  \boldsymbol{\Psi}$. Then, its norm is 
\[
\|\phi\|_{H_G}^2 = \sum_{k=1}^n\lambda_k^{-1} \big(x^\top V^{-1}\big)_k^2 = x^\top V^{-1} \Lambda^\dagger V^{-\top} x = x^\top (V\Lambda V^\top )^\dagger x .  
\]
Thus, $C_{rkhs} = (V\Lambda V^\top )^\dagger = B (A^\top A)^\dagger B$ and $C_{rkhs}^{\dagger}= V\Lambda V^\top = B^{-1}(A^\top A) B^{-1}$. 
\end{proof}
% \fi 

In particular, when \cref{eq:Ab} is either a discrete model or a discretization of \cref{eq:FIE} based on Riemann sum approximation of the integral, the exploration measure $\rho$ is the normalized column sum of the absolute values of the matrix $A$, and $B=\mathrm{diag}(\rho)$. See \Cref{sec:FIE_num} for details. 

%%%%%%%%

\subsection{DARTR: data-adaptive RKHS Tikhonov regularization} 
We review DARTR, a data-adaptive RKHS Tikhonov regularization algorithm introduced in \cite{LLA22}. 

Specifically, it solves the problem \cref{eq:Ab} with regularization:  
% (when it is written in the form \cref{eq:FIE}-\cref{eq:A_basis}): 
\[
(\widehat x_{\lambda_*}, \lambda_*) = \argmin_{x\in \R^n,\lambda\in \R^+} \calE_\lambda(x), \quad \text{ where }  \calE_\lambda(x):= \|Ax-b\|^2 + \lambda \|x\|_{C_{rksh}}^2,
\]
where the norm $\| \cdot \|_{C_{rksh}}$ is the DA-RKHS norm introduced in \Cref{thm:RKHS_norm}. % Let $\bA= A^\top A$ and $\bb = A^\top b$. 
A direct solution minimizing $\calE_\lambda(x)$ is to solve $(A^\top A + \lambda  C_{rkhs}) x_\lambda =A^\top b$. 
However, the computation of $ C_{rkhs}$ requires a pseudo-inverse that may cause numerical instability. 

DARTR introduces a transformation matrix $C_* := V \Lambda^{1/2} $ to avoid using the pseudo-inverse.   Note that $C_*^\top   C_{rkhs} C_* = \begin{pmatrix}I_r & 0 \\ 0 & 0  \end{pmatrix} := \mathbf{I}_r$, where $I_r$ is the identity matrix with rank $r$, the number of positive eigenvalues in $\Lambda$. Then, the linear equation $(A^\top A + \lambda  C_{rkhs}) x_\lambda =A^\top b$ is equivalent to 
\begin{equation}\label{eq:lsqminnorm}
(C_*A^\top A C_*+ \lambda \mathbf{I}_r) \widetilde x_\lambda = C_* A^\top b
\end{equation}
 with $\widetilde x_\lambda= C_*^{-1} x_\lambda$. Thus, DARTR computes $\widetilde x_\lambda$ in the above equation by least squares with minimal norm, and returns $x_\lambda= C_* \widetilde x_\lambda$. 
 
% Algorithm \ref{alg:dartr} summarizes the implementation of DARTR.  
DARTR is a direct method based on matrix decomposition, and it takes $O(mn^2+n^3)$ flops. Hence, it is computationally infeasible when $n$ is large. The iterative method in the next section implements the RKHS regularization in a scalable fashion.

%%%%%%%%%%-------------------------------------------------------
%%%%%%%%%%-------------------------------------------------------
\section{Iterative Regularization by DA-RKHS}\label{sec:alg-idartr}
This section introduces a subspace project method tailored to utilize the DA-RKHS for iterative regularization. As an iterative method, it achieves scalability by accessing the coefficient matrix only via matrix-vector multiplications, producing a sequence of estimators until reaching a desired solution. This section follows the notation conventions in \Cref{tab:iDartr}.  \vspace{-3mm}
 \begin{table}[tbh!] 
  \caption{Table of notations.  }  \label{tab:iDartr}\vspace{-3mm}
  \centering
\begin{tabular}{ c ll }
\toprule
% \multirow{2}{*}{Continuous model} & $y_l=L_l \vecphi$, $\vecphi\in L^2_\rho$, $y_l\in \spaceY$  \\    &  \\
% \hline  
%Discrete model & $ b = A \vecphi + \bw$, $\vecphi\in \R^n$, $\obs\in  \R^m$  \\  
\hline  $A, B, C $ & matrix or array by capital letters\\
\hline $b, c, x,y,z,u,v$ & vector by regular letters \\
\hline $\alpha,\beta,\gamma$ & scalar by Greek letters\\ 
\hline $\mathcal{R}(A)$ and $\mathcal{N}(A)$ & the range and null spaces of matrix $A$ \\
% \hline \multirow{1}{*}{Direct method} &  $\bA= A^\top A$, $\bb=A^\top\obs$  \\                                              &  \\
% \hline \multirow{2}{*}{Iterative method} & \\                                                     &  \\                  
\bottomrule
\end{tabular}
\end{table}

%%--------------------------------------------------
\subsection{Overview}
Our regularization method is based on subspace projection in the DA-RKHS. 
It iteratively constructs a sequence of linear subspaces $\mathcal{S}_k$ of the DA-RKHS $(\mathbb{R}^{n},\langle \cdot , \cdot \rangle_{ C_{rkhs}})$, and recursively solves projected problems 
\begin{equation}\label{subspace_regu}
	x_k =\argmin_{\vecphi\in\mathcal{X}_k}\|\vecphi\|_{ C_{rkhs}}, \ \ \mathcal{X}_k = \{\vecphi: \min_{\vecphi\in\mathcal{S}_k}\|A\vecphi-\obs\|_{2}\}. 
\end{equation}
This process yields a sequence of solutions $\{\vecphi_k\}$, each emerging from its corresponding subspace.  We ensure the uniqueness of the solution within each iteration, as detailed in \Cref{thm:solu_uniqueness}.  The iteration proceeds until it meets an early stopping criterion, designed to exclude excessive noisy components and thereby achieve effective regularization. The spaces $\mathcal{S}_k$ are called the \textit{solution subspaces}, and the iteration number $k$ plays the role of the regularization parameter. 

\Cref{alg:iDartr} outlines this procedure, which is a recursion of the following three parts. 
\begin{enumerate}
\item[P1] \emph{Construct the solution subspaces.} We introduce a new generalized Golub-Kahan bidiagonalization (gGKB) to construct the solution subspaces in the DA-RKHS iteratively. The procedure is presented in \Cref{alg:gGKB}.

\item[P2] \emph{Recursively update the solution to the projected problem.} We solve the least squares problem in the solution subspaces in \cref{subspace_regu} efficiently by a new LSQR-type algorithm, \Cref{alg:update}, updating $\|x_k\|_{C_{rkhs}}$ and the residual norm  $\|A\vecphi-\obs\|_2$ without even computing the residual. 

\item[P3] \emph{Regularize by early stopping.} We select the optimal $k$ by either the discrepancy principle (DP) when we have an accurate estimate of $\|\bw\|_2$ or the L-curve criterion otherwise.  
\end{enumerate}
\begin{algorithm}[htb]
	\caption{iDARR: Iterative Data-Adaptive RKHS Regularization}\label{alg:iDartr}
 \algorithmicrequire\ $A\in\mathbb{R}^{m\times n}$, $b\in\mathbb{R}^{m}$, $B=\mathrm{diag}(\rho)$, $\vecphi_{0}=\mathbf{0}$, $\bar{\vecphi}_0=\mathbf{0}$, where $\rho$ is the exploration measure in \cref{eq:exp_measure}.  
	\begin{algorithmic}[1]
		\For{$k=1,2,\ldots,$}
		\State  P1. Compute $u_k$, $z_k$, $\bar{z}_k$, $\alpha_k$, $\beta_k$ by \texttt{gGKB} in \Cref{alg:gGKB}. 
		\State P2. Update $\vecphi_k$, $\bar{\gamma}_{k+1}$, $\|\vecphi_k\|_{ C_{rkhs}}$, etc. by \Cref{alg:update}. 
		\State P3: Stop at iteration $k_*$ if \textit{Early stopping criterion} is satisfied.   \Comment{L-curve or DP}
		% \State Terminate at the estimated iteration $k_*$.
		\EndFor
	\end{algorithmic}
\algorithmicensure\ Final solution $\vecphi_{k_*}$  
\end{algorithm}

 In the next subsection, we present details for these key parts. Then, we analyze the computational complexity of the algorithm. 
 
\subsection{Algorithm details and derivations}
This subsection presents detailed derivations for the three parts P1-P3 in \Cref{alg:iDartr}. 

\noindent\textbf{P1. Construct the solution subspaces.} 
 We construct the solution subspaces by elaborately incorporating the regularization term $\|\cdot\|_{ C_{rkhs}}^2$ in the Golub-Kahan bidiagonalization (GKB) process. A key point is to use $ C_{rkhs}^{\dagger} =B^{-1}A^\top AB^{-1} $ to avoid explicitly computing $ C_{rkhs}$, which involves the costly spectral decomposition of the normal operator, see \cref{eq:mat_rkhs} in \Cref{thm:RKHS_norm}. 
 
%   Recall that the matrix $ C_{rkhs}$ in \cref{eq:mat_rkhs} is positive semi-definite, and its role is to restrict the search to be inside the RKHS $H_G$, whose closure is the FSOI--- XXX. 
 
Consider first the case where $ C_{rkhs}$ is positive definite. In this scenario, $A$ has full column rank, and the $ C_{rkhs}$-inner product Hilbert space $(\mathbb{R}^{n}, \langle \cdot , \cdot \rangle_{ C_{rkhs}})$ is a discrete representation of the RKHS $H_G$ with the given basis functions. Note that the true solution is mapped to the noisy $\obs$ by $A: (\mathbb{R}^{n}, \langle \cdot , \cdot \rangle_{ C_{rkhs}}) \to (\mathbb{R}^{m}, \langle \cdot , \cdot \rangle_2)$. 
Let $A^{*}: (\mathbb{R}^{m}, \langle \cdot , \cdot \rangle_2) \to (\mathbb{R}^{n}, \langle \cdot , \cdot \rangle_{ C_{rkhs}})$ be the adjoint of $A$, i.e. $\langle A\vecphi, \obs \rangle_2=\langle \vecphi, A^{*}\obs \rangle_{ C_{rkhs}}$ for any $\vecphi\in\mathbb{R}^{n}$ and $\obs\in \mathbb{R}^{m}$. By definition, the matrix-form expression of $A^{*}$ is  
\begin{equation}\label{eq:A*}
	A^{*} = C_{rkhs}^{-1}A^\top, 
\end{equation}
since $\langle A \vecphi, \obs \rangle_{2}=\vecphi^\top  A^\top\obs$ and $\langle \vecphi, A^{*}\obs \rangle_{ C_{rkhs}}=\vecphi^\top  C_{rkhs}A^{*}\obs$ for any $\vecphi$ and $\obs$.

The Golub-Kahan bidiagonalization (GKB) process recursively constructs orthonormal bases for these two Hilbert spaces starting with the vector $\obs$ as follows: 
\begin{subequations}{\label{eq:GKB}}
\begin{align}
	& \beta_1u_1 = \obs, \ \ \alpha_1z_1 = A^{*}u_1,  \label{GKB1} \\
	& \beta_{i+1}u_{i+1} = Az_i - \alpha_iu_i, \label{GKB2} \\
	& \alpha_{i+1}z_{i+1} = A^{*}u_{i+1} - \beta_{i+1}z_i, \label{GKB3}
	\end{align}
\end{subequations}
where $u_i\in(\mathbb{R}^{m}, \langle \cdot , \cdot \rangle_2)$, $z_i\in(\mathbb{R}^{n}, \langle \cdot , \cdot \rangle_{ C_{rkhs}})$ with $z_0=\boldsymbol{0}$, and $\alpha_i$, $\beta_i$ are normalizing factor such that $\|u_i\|_2=\|z_{i}\|_{ C_{rkhs}}=1$. The iteration starts with $u_1=\obs/\beta_1$ with $\beta_1=\|\obs\|_2$. Using $A^*$ in \cref{eq:A*}, we write \cref{GKB3} as 
\begin{align}\label{GKB4}
	\alpha_{i+1}z_{i+1}
	=  C_{rkhs}^{-1}A^\top u_{i+1}-\beta_{i+1}z_i 
	% = B^{-1}A^\top LB^{-1}A^\top u_{i+1}-\beta_{i+1}z_i 
\end{align}
with $\alpha_{i+1} = \| C_{rkhs}^{-1}A^\top  u_{i+1}-\beta_{i+1}z_{i}\|_{ C_{rkhs}}$. To compute $\alpha_{i+1}$ without explicitly computing $ C_{rkhs}$, define $\bar{z}_i= C_{rkhs}z_i$. Then we have
\begin{align}\label{GKB5}
	\alpha_{i+1}\bar{z}_{i+1} = A^{\top}u_{i+1}-\beta_{i+1}\bar{z}_i,
\end{align}
where $\bar{z}_0:=\boldsymbol{0}$. Let $p=A^{\top}u_{i+1}-\beta_{i+1}\bar{z}_i$. Then we obtain $\alpha_{i+1}=\| C_{rkhs}^{-1}p\|_{ C_{rkhs}}=(p^{\top} C_{rkhs}^{-1}p)^{1/2}$, which uses $ C_{rkhs}^{-1}=B^{-1}A^\top A B^{-1}$ without computing $C_{rkhs}$.  

 Next, consider that $ C_{rkhs}$ is positive semidefinite. The iterative process remains the same with $ C_{rkhs}^{-1}$ replaced by the pseudo-inverse $ C_{rkhs}^{\dag}$, because $ C_{rkhs}^{\dag}=B^{-1}A^\top A B^{-1}$ has the same form as $ C_{rkhs}^{-1}$ for the non-singular case. Specifically, the recursive relation \cref{GKB4} becomes
\begin{align}\label{GKB4m}
	\alpha_{i+1}z_{i+1}	=  C_{rkhs}^{\dag}A^\top u_{i+1}-\beta_{i+1}z_i. 
\end{align}
 To compute $\alpha_{i+1}$, we use the property that $z_i\in\mathcal{R}( C_{rkhs})$, which will be proved in \Cref{prop:z_in_RKHS}. Note that $ C_{rkhs}^{\dag} C_{rkhs}=P_{\mathcal{N}( C_{rkhs})^{\perp}}=P_{\mathcal{R}( C_{rkhs})}$ since $ C_{rkhs}$ is symmetric, where $P_{\mathcal{X}}$ is the projection operator onto subspace $\mathcal{X}$. It follows that $ C_{rkhs}^{\dag} C_{rkhs}z_i=z_i$. Therefore, \cref{GKB4m} becomes
\begin{equation*}% \label{GKB5m}
	\alpha_{i+1}z_{i+1}
	=  C_{rkhs}^{\dag}(A^\top u_{i+1}-\beta_{i+1} C_{rkhs}z_i).
\end{equation*}
Letting $\bar{z}_i= C_{rkhs}z_i$ and $p=A^{\top}u_{i+1}-\beta_{i+1}\bar{z}_i$ again, we get $\alpha_{i+1}=\| C_{rkhs}^{\dag}p\|_{ C_{rkhs}}=(p^{\top} C_{rkhs}^{\dag}p)^{1/2}$.

Thus, the two cases of $ C_{rkhs}$ lead to the same iterative process. We summarize the iterative process in \Cref{alg:gGKB}, and call it \textit{generalized Golub-Kahan bidiagonalization} (gGKB). 
\begin{algorithm}[htb]
	\caption{Generalized Golub-Kahan bidiagonalization (gGKB)}\label{alg:gGKB}
 	\algorithmicrequire\ $A\in\mathbb{R}^{m\times n}$, $\obs\in\mathbb{R}^{m}$, $B=\mathrm{diag}(\rho)$
	\begin{algorithmic}[1]
		\State Initialize: let $\beta_1=\|b\|_2$, \ $u_1=b/\beta_1$, and compute $p=A^\top u_1$, \, $s=B^{-1}A^\top AB^{-1}p$. 
		\State Let $\alpha_1 = (s^\top p)^{1/2}$, \ $z_1=s/\alpha_1$, \ $\bar{z}_1=p/\alpha_1$.  
		\For {$i=1,2,\dots,k,$}
		\State $r=Az_i-\alpha_iu_i$, \, $\beta_{i+1}=\|r\|_2$, \ $u_{i+1}=r/\beta_{i+1}$; 
		\State $p=A^\top u_{i+1}-\beta_{i+1}\bar{z}_i$, \,  %$p=p-\beta_{i+1}\bar{z}_i$
               $s=B^{-1}A^\top A B^{-1}p$;  \Comment{ $C_{rkhs}^\dag =B^{-1}A^\top A B^{-1}$ }
		\State $\alpha_{i+1}= (s^\top p)^{1/2}$, \ $z_{i+1}=s/\alpha_{i+1}$, \ $\bar{z}_{i+1}=p/\alpha_{i+1}$.  \Comment{$\bar{z}_i= C_{rkhs}z_i$ \quad\quad\quad\quad\quad\quad }
		\EndFor
	\end{algorithmic}
 	\algorithmicensure\ $\{\alpha_i, \beta_i\}_{i=1}^{k+1}$, \ $\{u_i, z_i, \bar{z}_i\}_{i=1}^{k+1}$
\end{algorithm}

Suppose the gGKB process terminates at step $k_{t} := \max_{k\geq 1}\{\alpha_{k}\beta_{k}>0\}$. We show in \Cref{prop:onb-GKB}--\Cref{prop:gGKB_iter_num} that the output vectors $\{u_i\}_{i=1}^{k_t}$ and  $\{z_i\}_{i=1}^{k_t}$ are orthonormal in $(\mathbb{R}^{m}, \langle \cdot , \cdot \rangle_2)$ and $(\mathbb{R}^{n}, \langle \cdot , \cdot \rangle_{ C_{rkhs}})$, respectively. In particular, they span two Krylov subspaces generated by $\{A C_{rkhs}^{\dag}A^{\top}, b\}$ and $\{C_{rkhs}^{\dag}A^{\top}A,  C_{rkhs}^{\dag}A^{\top}b\}$, respectively.

In matrix form, the $k$-step gGKB process with $k<k_t$, starting with vector $b$, produces a $2$-orthonormal matrix $U_{k+1}=(u_1,\dots,u_{k+1})\in \R^{m\times (k+1)}$ (i.e., $U_{k+1}^\top U_{k+1} = I_{k+1}$), a $ C_{rkhs}$-orthonormal matrix $Z_{k+1}=(z_1,\dots,z_{k+1})\in \R^{n\times (k+1)}$, and a bidiagonal matrix 
\begin{equation}
	B_{k}=\begin{pmatrix}
		\alpha_{1} & & & \\
		\beta_{2} &\alpha_{2} & & \\
		&\beta_{3} &\ddots & \\
		& &\ddots &\alpha_{k} \\
		& & &\beta_{k+1}
		\end{pmatrix}\in  \mathbb{R}^{(k+1)\times k}
\end{equation}
such that the recursion in \cref{GKB1},\cref{GKB2} and \cref{GKB4m} can be written as 
\begin{subequations}{\label{eq:GKB_matForm}}
\begin{align}
	& \beta_1U_{k+1}e_{1} = b, \label{GKB6} \\
	& AZ_k = U_{k+1}B_k, \label{GKB7} \\
	&  C_{rkhs}^{\dag}A^\top U_{k+1} = Z_kB_{k}^{T}+\alpha_{k+1}z_{k+1}e_{k+1}^\top , \label{GKB8}
\end{align}
\end{subequations}
where $e_1$ and $e_{k+1}$ are the first and $(k+1)$-th columns of $I_{k+1}$.  
We emphasize that $B_k$ is a bidiagonal matrix of full column rank, since $\alpha_i, \beta_i>0$ for all $i\leq k+1$. 

%%---------
\noindent\textbf{P2. Recursively update the solution to the projected problem.}  
For each $k\leq k_t$, i.e., before the gGKB terminates, we solve \cref{subspace_regu} in the subspace 
\begin{equation}\label{eq:Sk}
\mathcal{S}_k:=\mathrm{span}\{z_1,\dots,z_k\}
\end{equation} 
%($\subseteq \mathcal{R}( C_{rkhs})$ recall that $z_i\in \mathcal{R}( C_{rkhs}) $ by Proposition \ref{prop:z_in_RKHS}) 
and compute the RKHS norm of the solution.

We show first the uniqueness of the solution to \cref{subspace_regu}. 
From \cref{GKB7} we have $B_k=U_{k+1}^\top AZ_k$, which implies that $B_k$ is a projection of $A$ onto the two subspaces $\mathrm{span}\{U_{k+1}\}$ and $\mathrm{span}\{Z_k\}$. Since any vector $\vecphi$ in $\mathcal{S}_k$ can be written as $\vecphi=Z_k y $ with a $ y \in \R^k$, we obtain from \cref{GKB6} and \cref{GKB7} that for any $\vecphi=Z_k y$, 
 \begin{equation}\label{eq:min-Sk}
\begin{aligned}
%	\|A\vecphi-\obs\|_2 
%	=\|AZ_k y -U_{k+1}\beta_1e_{1}\|_2 
%	=\|U_{k+1}B_k y -U_{k+1}\beta_1e_{1}\|_2
%	= \|B_k y -\beta_1e_{1}\|_2.
     \min_{\vecphi=Z_k y }\|A\vecphi-\obs\|_2 
	 &= \min_{ y \in\mathbb{R}^{k}}\|AZ_k y -U_{k+1}\beta_1e_{1}\|_2 \\
	 &= \min_{ y \in\mathbb{R}^{k}}\|U_{k+1}B_k y -U_{k+1}\beta_1e_{1}\|_2 
      = \min_{ y \in\mathbb{R}^{k}}\|B_k y -\beta_1e_{1}\|_2.
\end{aligned}
\end{equation}
Since $k\leq k_t$, $B_k$ has full column rank, this $k$-dimensional least squares problem has a unique solution. 
%As a result, $\mathcal{X}_k = \{\vecphi: \min_{\vecphi\in\mathcal{S}_k}\|A\vecphi-\obs\|_{2}\}$ has a unique element $x_k= Z_k y _k$.
Therefore, the unique solution to \cref{subspace_regu} is
\begin{equation}\label{eq:solu_in_Rk}
	\vecphi_k = Z_k y _k, \ \  y _k=\argmin_{ y \in\mathbb{R}^{k}}\|B_k y -\beta_1e_{1}\|_2. 
\end{equation}
We note that the above uniqueness is independent of the specific construction of the basis vectors $\{z_k\}$. In general, as long as $\mathcal{S}_k\subseteq \mathcal{R}( C_{rkhs})$, the solution to \cref{subspace_regu} is unique; see \Cref{thm:solu_uniqueness}. 

Importantly, we recursively compute $x_k$ without explicitly solving $y_k$ in \cref{eq:solu_in_Rk}, avoiding the $O(k^3)$ computational cost. To this end, we adopt a procedure similar to the LSQR algorithm in \cite{Paige1982} to iteratively update $\vecphi_k$ from $\vecphi_0=\mathbf{0}$. It starts from the following Givens QR factorization: 
\[Q_k\begin{pmatrix}
		B_k & \beta_{1}e_{1}
	\end{pmatrix}
	=\begin{pmatrix}
		R_k & f_k \\
		    & \bar{\gamma}_{k+1}
	\end{pmatrix}: =
	\left(\begin{array}{ccccc : c}		
		\rho_1& \theta_2 &  &  &  & \gamma_1 \\		
		& \rho_2 & \theta_3 & & & \gamma_2 \\		
		& & \ddots & \ddots & & \vdots \\ 		
		& & & \rho_{k-1} & \theta_k & \gamma_{k-1} \\		
		& & &  & \rho_k & \gamma_k \\	\hdashline	
		& & &  &  & \bar{\gamma}_{k+1}	
	\end{array}\right) ,\]
where the orthogonal matrix $Q_k$ is the product of a series of Givens rotation matrices, and $R_k$ is a bidiagonal upper triangular matrix; see \cite[\S 5.2.5]{Golub2013}. We implement the Givens QR factorization using the procedure in \cite{Paige1982}, which recursively zeros out the subdiagonal elements $\beta_{i}$ for each $2\leq i\leq k+1$. Specifically, at the $i$-th step, a Givens rotation zeros out $\beta_{i+1}$ by
\[\begin{pmatrix}
	c_i & s_i \\
	s_i & -c_i
\end{pmatrix}
\begin{pmatrix}
	\bar{\rho}_i & 0 & \bar{\gamma}_i  \\
	\beta_{i+1} & \alpha_{i+1} & 0
\end{pmatrix}=
\begin{pmatrix}
	\rho_i & \theta_{i+1} & \gamma_i \\
	0 & \bar{\rho}_{i+1}  & \bar{\gamma}_{i+1}
\end{pmatrix},\]
where the entries $c_i$ and $s_i$ of the Givens rotations satisfy $c_i^2+s_i^2=1$, and the elements $\rho_i$, $\theta_{i+1}$, $\bar{\rho}_{i+1}$, $\gamma_i$, $\bar{\gamma}_{i+1}$ are recursively updated accordingly; see \cite{Paige1982} for more details. 

As a result of the QR factorization, we can write 
\begin{equation}\label{res1}
	\|B_k y -\beta_1e_{1}\|_{2}^2=\Big\|Q_k\begin{pmatrix}
		B_k & \beta_{1}e_{1}
	\end{pmatrix}\begin{pmatrix}
		 y  \\ -1
	\end{pmatrix}\Big\|_{2}^2=\|R_k y -f_k\|_{2}^2+|\bar{\gamma}_{k+1}|^2. 
\end{equation}
Hence, the solution to $\min_{ y \in\mathbb{R}^k}\|B_k y -\beta_{1}e_{1}\|$ is $ y _k=R_{k}^{-1}f_k$. Factorizing $R_k$ as
\[R_k=D_k\bar{R}_k, \ \ 
D_k :=\begin{pmatrix}
	\rho_1 & & & \\
	& \rho_2 & & \\
	& & \ddots & \\
	& & & \rho_k
\end{pmatrix}, \ 
\bar{R}_k :=\begin{pmatrix}
	1 & \theta_2/\rho_1 & & \\
	& 1 & \theta_3/\rho_2 & \\
	& & \ddots & \theta_k/\rho_{k-1} \\
	& & & 1
\end{pmatrix}\]
and denoting $W_k=Z_k\bar{R}_{k}^{-1}=(w_1,\dots,w_k)$, we get
\begin{equation*}\label{up1_x}
	\vecphi_k=Z_k  y _k=Z_k R_{k}^{-1}f_k=(Z_k\bar{R}_{k}^{-1})(D_{k}^{-1}f_k)= W_k (D_{k}^{-1}f_k) = \sum_{i=1}^k (\gamma_{i}/\rho_{i})w_i.
\end{equation*}
Updating $x_k$ recursively, and solving $W_k\bar{R}_k=Z_k$ by back substitution, we obtain:  
\begin{equation}\label{up2_x}
	\vecphi_{i}=\vecphi_{i-1}+(\gamma_{i}/\rho_{i})w_{i}, \ \ \  
	w_{i+1}=z_{i+1}-(\theta_{i+1}/\rho_{i})w_{i}, \quad \forall 1\leq  i\leq k.
\end{equation}

Lastly, we compute $\|\vecphi_i\|_{ C_{rkhs}}$ without an explicit $C_{rkhs}$. We have from \cref{up2_x}: 
\begin{equation*}
	 C_{rkhs}\vecphi_{i}= C_{rkhs}\vecphi_{i-1}+(\gamma_{i}/\rho_{i}) C_{rkhs}w_{i}, \ \ 
	 C_{rkhs}w_{i+1}= C_{rkhs}z_{i+1}-(\theta_{i+1}/\rho_{i}) C_{rkhs}w_{i}.
\end{equation*}
Letting $\bar{\vecphi}_i= C_{rkhs}\vecphi_i$ and $\bar{w}_i= C_{rkhs}w_i$, and recalling that $\bar{z}_{i}= C_{rkhs}z_{i}$, we have
\begin{equation}
\|\vecphi_i\|_{ C_{rkhs}}^2=\vecphi_{i}^{\top}\bar{\vecphi}_i,\quad 	\bar{\vecphi}_{i}=\bar{\vecphi}_{i-1}+(\gamma_{i}/\rho_{i})\bar{w}_{i}, \ \ \  
	\bar{w}_{i+1}=\bar{z}_{i+1}-(\theta_{i+1}/\rho_{i})\bar{w}_{i}.
\end{equation}

\begin{algorithm}[htb]
	\caption{Updating procedure}\label{alg:update}
	\begin{algorithmic}[1]
		\State {Let $\vecphi_{0}=\mathbf{0},\ \bar{\vecphi}_0=\mathbf{0}, \ w_{1}=z_{1}, \ \bar{w}_{1}=\bar{z}_1, \ \bar{\gamma}_{1}=\beta_{1},\ \bar{\rho}_{1}=\alpha_{1}$}
		\For{$i=1,2,\ldots,k,$}
		\State $\rho_{i}=(\bar{\rho}_{i}^{2}+\beta_{i+1}^{2})^{1/2}$, $c_{i}=\bar{\rho}_{i}/\rho_{i},\ s_{i}=\beta_{i+1}/\rho_{i}$
		\State$\theta_{i+1}=s_{i}\alpha_{i+1},\ \bar{\rho}_{i+1}=-c_{i}\alpha_{i+1}$, $\gamma_{i}=c_{i}\bar{\gamma}_{i},\ \bar{\gamma}_{i+1}=s_{i}\bar{\gamma}_{i} $
		\State $\vecphi_{i}=\vecphi_{i-1}+(\gamma_{i}/\rho_{i})w_{i} $, $w_{i+1}=z_{i+1}-(\theta_{i+1}/\rho_{i})w_{i}$
		\State $\bar{\vecphi}_{i}=\bar{\vecphi}_{i-1}+(\gamma_{i}/\rho_{i})\bar{w}_{i}$, $\bar{w}_{i+1}=\bar{z}_{i+1}-(\theta_{i+1}/\rho_{i})\bar{w}_{i}$
		\State $\|\vecphi_i\|_{ C_{rkhs}}=(\vecphi_{i}^{\top}\bar{\vecphi}_i)^{1/2}$
		\EndFor
	\end{algorithmic}
\end{algorithm}

 The whole updating procedure is described in \Cref{alg:update}. 
This algorithm yields the residual norm $\|A \vecphi_k-\obs\|_2$ without explicitly computing the residual. In fact, by \cref{eq:min-Sk} and \cref{res1} we have 
\begin{equation}\label{res2}
	\bar{\gamma}_{k+1}=\|B_k y _{k}-\beta_1 e_{1}\|_2 = \|A \vecphi_k-\obs\|_2.
\end{equation}
Note that $\bar{\gamma}_{k+1}$ decreases monotonically since $\vecphi_k$ minimizes $\|A\vecphi-\obs\|_2$ in the gradually expanding subspace $\mathrm{span}\{Z_k\}$.

Importantly, \Cref{alg:update} efficiently computes the solution $x_k$ and $\|x_k\|_{C_{rkhs}}$. At each step of updating $x_i$ or $w_{i+1}$, the computation takes $O(2n)$ flops. Similarly, for updating $\bar{x}_i$ or $\bar{w}_{i+1}$, as well as for computing $\|\vecphi_i\|_{ C_{rkhs}}$, the number of flops are also $O(2n)$. Therefore, the dominant computational cost is $O(10n)$. In contrast, if $y_k$ is solved explicitly at each step, it takes $O(\sum_{i=1}^{k}i^3)\sim O(k^4)$ flops; together with the step of forming $x_k=Z_ky_k$ that takes $O(kn)$ flops, they lead to a total cost of $O(kn+k^4)$ flops. Thus, the LSQR-type iteration in \Cref{alg:update} significantly reduces the number of flops from $O(kn+k^4)$ to $O(10n)$.

\noindent\textbf{P3. Regularize by early stopping.} 
An early stopping strategy is imperative to prevent the solution subspace from becoming excessively large, which could otherwise compromise the regularization. This necessity is rooted in the phenomenon of \emph{semi-convergence}: the iteration vector $x_k$ initially approaches an optimal regularized solution but subsequently moves towards the unstable naive solution to $\min_{x\in\mathcal{R}( C_{rkhs})}\|Ax-b\|_2$, as detailed in \Cref{naive}.

For early stopping, we adopt the L-curve criterion, as outlined in \cite{gazzola2019ir}. This method identifies the ideal early stopping iteration $k_*$ at the corner of the curve represented by 
\begin{equation}
	\left(\log\|A\vecphi_{k}-\obs\|_2, \|\vecphi_k\|_{ C_{rkhs}}\right) = \left(\log\bar{\gamma}_{k+1}, \log\|\vecphi_k\|_{ C_{rkhs}}\right). 
\end{equation}
Here $\bar{\gamma}_{k+1}$ and $\|\vecphi_k\|_{ C_{rkhs}}$ are computed with negligible cost in \Cref{alg:update}. To construct the L-curve effectively, we set the gGKB to execute at least 10 iterations. Additionally, to enhance numerical stability, we stop the gGKB when either $\alpha_i$ or $\beta_i$ is near the machine precision, as inspired by \Cref{naive}.  

It is noteworthy that the \emph{discrepancy principle} (DP) presents a viable alternative when the measurement error $\|\bw\|_2$ in \cref{eq:Ab} is available with a high degree of accuracy. The DP halts iterations at the earliest instance of $k$ that satisfies 
% \begin{equation}\label{discrepancy}
$	\bar{\gamma}_{k+1}=\|Ax_k-b\|_2 \leq \tau\|\bw\|_2,
$ % \end{equation}
where $\tau$ is chosen to be marginally greater than 1.

% \bigskip

\subsection{Computational Complexity} \label{subsec:complexity}
Suppose the algorithm takes $k$ iterations and the basis matrix $B$ is diagonal. Recall that $A\in \R^{m\times n}$, $B\in \R^{n\times n}$, $u_i\in\R^m$ and $z_i\in \R^n$. The total computational cost of \Cref{alg:iDartr} is about $O(3mnk)$ when $m\leq n/3$ or $k<n/3$;  and about $O((m+k)n^2)$ when otherwise.  The cost is dominated by the gGKB process since the cost of the update procedure in \Cref{alg:update} is only $O(n)$ at each step. 

The gGKB can be computed in two approaches. The first approach uses only matrix-vector multiplication. The main computations in each iteration of gGKB occur at the matrix-vector products $p=A^{\top}u_i$ and $s=B^{-1}A^{\top}AB^{-1}p = B^{-1}(A^{\top}(A(B^{-1}p)))$, which take $O(mn)$ and $O(2mn)$ flops respectively. Thus, the total computational cost of gGKB is $O(3mnk)$ flops. Another approach is using $\bA=A^{\top}A$ instead of $A^{\top}$ and $A$ to compute $s$. In this approach, the computation of $\bA$ from $A$ takes $O(mn^2)$ flops, and the matrix-vector multiplication $\bA v$ in each iteration takes about $O(n^2)$ flops. Hence, the total cost of $k$ iterations is $O(mn^2+kn^2)$. The second approach is faster when $mn^2+n^2k < 3mnk$, or equivalently, $(3m-n)k>mn$. That is, roughly speaking, $m>n/3$ and $k>mn/(3m-n)>n/3$.

% We summarize the computation cost in Table \ref{tab:computCost}. 
In practice, the matrix-vector computation is preferred since the iteration number $k$ is often small. The resulting iDARR algorithm takes about $O(3mnk)$ flops. 
% \begin{table}[h]
%  \centering
% \begin{tabular}{ | l | c | c |} \hline
%  \textbf{Computation} & \textbf{Computation cost} & \textbf{Suitable size}\\ 
% \hline
%   matrix-vector & $ O(3mnk)$  & $m\leq n/3$ or $k\leq n/3$ \\ 
%   $\bA= A^\top A$ &  $ O((m+k)n^2)$  & $m> n/3$ and $k>n/3$ \\
% \hline
% \end{tabular}
% \caption{Computational complexity of the iDARR Algorithm \ref{alg:iDartr} when $B$ is diagonal.}
% \end{table}

%% %% $B$ is non-diagonal: not considered in this study. 
% Additionally, when $B$ is non-diagonal, as often happens in regression \cite{LangLu22,LAY22}, the computation will be more expensive. For large $n$, there are two cases to be considered. If $B$ has a special structure such as sparsity, we can exploit the sparse Cholesky factorization to compute $B^{-1}$ in advance, which takes $O(n^3)$. Then, computing $s=B^{-1}A^{\top}AB^{-1}p$ and $p=A^\top u_i$ for $k$ iterations takes $O(2mn+2n^2)$ and $O(mnk)$ flops, followed the above matrix-vector strategy. Thus, the total cost is of order $O(n^3+2mn+2n^2)$.  Otherwise, it would be better to avoid computing $B^{-1}$ directly by computing $s= B^{-1}(A^{\top}A(B^{-1}p))$, where the matrix-vector product with $B^{-1}$ is approximated by iteratively solving a linear system with matrix $B$. A future direction is to avoid the direct computation of $B^{-1}$ and reduce the cost by computing through solving $Bw=v$ with an adaptive accuracy level.

%%-------------------------------------------------------
\section{Properties of gGKB}\label{sec:gGKB}%he generalized Golub-Kahan bidiagonalization}
%%%%%%%-------------------------------------------------------
This section studies the properties of the gGKB in \Cref{alg:gGKB}, including the structure of the solution subspace, the orthogonality of the resulted vectors $\{u_i\}_{i=1}^{k+1}$ and $\{z_i\}_{i=1}^{k+1}$ and the number of iterations at termination, defined as:  
\begin{equation}\label{eq:gGK_niter}
\textbf{gGKB terminate step: }\quad  	k_{t}: = \max_{k\geq 1}\{\alpha_{k}\beta_{k}>0\}. 
\end{equation} 
Additionally, we show that the solution to \cref{subspace_regu} is unique in each iteration. 

Throughout this section, let $r$ denote the rank of $\Lambda$ and let $V_r$ denote the first $r$ columns of $V$, where $\Lambda$ and $V$ are matrices constituted by the generalized eigenvalues and eigenvectors of $\{A,B\}$ in \cref{eq:mat_rkhs} in \Cref{thm:RKHS_norm}. We have $\mathrm{rank}(A)=r$. Recall that $C_{rkhs} =  (V\Lambda V^\top)^{\dagger} = B (A^\top A)^\dagger B$. Note that the DA-RKHS is $( \mathcal{R}( C_{rkhs}),\langle \cdot , \cdot \rangle_{ C_{rkhs}})$.

 %%%% 
 \subsection{Properties of gGKB}% {RKHS-restricted Krylov subspaces}
 We show first that the gGKB-produced vectors $\{u_i\}_{i=1}^{k+1}$ and $\{z_i\}_{i=1}^{k+1}$ are orthogonal in $\R^n$ and in the DA-RKHS, and the solution subspaces of the gGKB are RKHS-restricted Krylov subspaces. 
 
 \begin{definition}[RKHS-restricted Krylov subspace] Let  $A\in \R^{m\times n}$ and $b\in \R^n$, and let $B\in\R^{n\times n}$ be a symmetric positive definite matrix. Let $C_{rkhs} = B(A^\top A)^\dagger B$, which defines an RKHS $(\mathbb{R}^{n},\langle \cdot , \cdot \rangle_{ C_{rkhs}})$. The \emph{RKHS-restricted Krylov subspaces} are % linear subspaces of $\R^n$ defined by 
\begin{equation}\label{space_z}
		\mathcal{K}_{k+1}( C_{rkhs}^{\dag}A^{\top}A,  C_{rkhs}^{\dag}A^{\top}b) = \mathrm{span}\{( C_{rkhs}^{\dag}A^{\top}A)^{i} C_{rkhs}^{\dag}A^{\top}b\}_{i=0}^{k},\quad k\geq 0. 
	\end{equation}
\end{definition}

The main result is the following. 
\begin{theorem}[Properties of gGKB] Recall $k_t$ in \cref{eq:gGK_niter}, and the gGKB generates vectors $\{u_i\}_{i=1}^{k}$, $\{z_i\}_{i=1}^{k}$ and $\mathcal{S}_k= \mathrm{span}\{z_i\}_{i=1}^k$. They satisfy the following properties: 
\begin{itemize}
\item[(i)] $\{u_i\}_{i=1}^{k}$ and $\{z_i\}_{i=1}^{k}$ are orthonormal in $\R^n$ and in $( \mathcal{R}( C_{rkhs}),\langle \cdot , \cdot \rangle_{ C_{rkhs}})$, respectively;  
\item[(ii)] 
$\mathcal{S}_k= \mathcal{K}_{k}( C_{rkhs}^{\dag}A^{\top}A,  C_{rkhs}^{\dag}A^{\top}b) $  for each $k\leq k_t$, and the termination iteration number is $k_t = \mathrm{dim}( \mathcal{K}_{\infty}( C_{rkhs}^{\dag}A^{\top}A,  C_{rkhs}^{\dag}A^{\top}b) )$. 
\end{itemize}
\end{theorem} 
\begin{proof} Part (i) follows from \Cref{prop:z_in_RKHS}, where we show that $z_k\in  \mathcal{R}( C_{rkhs}) $, and \Cref{prop:onb-GKB}, where we show the orthogonality of these vectors. 

For Part (ii), $\mathcal{S}_k= \mathcal{K}_{k}( C_{rkhs}^{\dag}A^{\top}A,  C_{rkhs}^{\dag}A^{\top}b) $ follows from that $\{z_i\}_{i=1}^{k}$ form an orthonormal basis of the RKHS-restricted Krylov subspace. We prove $k_t = \mathrm{dim}( \mathcal{K}_{\infty}( C_{rkhs}^{\dag}A^{\top}A,  C_{rkhs}^{\dag}A^{\top}b) )$ in \Cref{prop:gGKB_iter_num}. 
\end{proof}

\begin{proposition}\label{prop:z_in_RKHS}
For each $z_i$ generatad by gGKB in \cref{eq:GKB}, it holds that $z_i\in\mathcal{R}( C_{rkhs})$. Additionally, if $ q:=C_{rkhs}^{\dag}A^\top u_{i+1}-\beta_{i+1}z_i \neq \boldsymbol{0}$, then $\alpha_{i+1}=\|q\|_{ C_{rkhs}}\neq 0$. 
\end{proposition}
\begin{proof}
We prove it by mathematical induction. For $i=1$, we obtain from \cref{GKB4m} and \Cref{thm:RKHS_norm} that $\alpha_1z_1= C_{rkhs}^{\dag}A^\top u_{1}=V\Lambda V^{\top}A^\top u_{1}\in\mathcal{R}(V_r)=\mathcal{R}( C_{rkhs})$, where $r$ is the rank of $\Lambda$. Suppose $z_i\in\mathcal{R}( C_{rkhs})$ for $i\geq 1$. Using again \cref{GKB4m} and \Cref{thm:RKHS_norm} we get 
$$\alpha_{i+1}z_{i+1}= C_{rkhs}^{\dag}A^\top u_{i+1}-\beta_{i+1}z_i=V\Lambda V^{\top}A^\top u_{i+1}-\beta_{i+1}z_i\in \mathcal{R}( C_{rkhs}).$$
 Therefore, $z_{i+1}\in\mathcal{R}( C_{rkhs})$, and $q\in\mathcal{R}( C_{rkhs})$. 

% Note that  $q= C_{rkhs}^{\dag}A^\top u_{i+1}-\beta_{i+1}z_i\in\mathcal{R}( C_{rkhs})$, and $\alpha_{i+1}=\|q\|_{ C_{rkhs}}$. 
If $\alpha_{i+1}=0$, then $q\in\mathcal{N}( C_{rkhs})=\mathcal{R}( C_{rkhs})^{\perp}$. Therefore, 
%$q\in \mathcal{R}( C_{rkhs})\bigcap \mathcal{R}( C_{rkhs})^{\perp}=\{\boldsymbol{0}\}$, and then 
$q=\boldsymbol{0}$.
\end{proof}

Thus, even if $ C_{rkhs}$ is singular (positive semidefinite), the gGKB in \Cref{alg:gGKB} does not terminate as long as the right-hand sides of \cref{GKB2} and \cref{GKB4m} are nonzero, since the iterative computation of $\{\beta_{i+1}, u_{i+1}\}$ and $\{\alpha_{i+1}, z_{i+1}\}$ can continue. Next, we show that these vectors are orthogonal. 
%The next theorem shows that the solution subspaces generated by gGKB are the RKHS-restricted Krylov subspaces. XXX
\begin{proposition}[Orthogonality]\label{prop:onb-GKB}
Suppose the $k$-step gGKB does not terminate, i.e. $k< k_t$, then $\{u_i\}_{i=1}^{k+1}$ is a 2-orthonormal basis of the Krylov subspace
	\begin{equation}\label{space_u} 
			\mathcal{K}_{k+1}(A C_{rkhs}^{\dag}A^{\top}, b)=\mathrm{span}\{(A C_{rkhs}^{\dag}A^{\top})^{i}b\}_{i=0}^{k},
	\end{equation}
	and $\{z_i\}_{i=1}^{k+1}$ is a $C_{rkhs}$-orthonormal basis for the RKHS-restricted Krylov subspace in \cref{space_z}. 
\end{proposition}
\begin{proof}
Note from \Cref{thm:RKHS_norm} that $ C_{rkhs}^{\dag}=V\Lambda V^{\top}$. Let $W_r=V_{r}\Lambda_{r}^{1/2}$. Then $W_{r}^{\top} C_{rkhs}W_{r}=I_r$, $ C_{rkhs}^{\dag}=W_{r}W_{r}^{\top}$, and $\mathcal{R}(W_{r})=\mathcal{R}( C_{rkhs})$. For any $z_i$, \Cref{prop:z_in_RKHS} implies that there exists $v_i\in\mathbb{R}^{r}$ such that $z_i=W_{r}v_i$. We get from \cref{GKB2} and \cref{GKB4m} that
	\begin{align*}
		& \beta_{i+1}u_{i+1} = AW_{r}v_{i}-\alpha_{i}u_{i} , \\
		& \alpha_{i+1}v_{i+1} = W_{r}^{\top}A^{\top}u_{i+1}-\beta_{i+1}v_{i} , 
	\end{align*}
    where the second equation comes from $\alpha_{i+1}W_{r}v_{i+1} = W_{r}W_{r}^{\top}A^{\top}u_{i+1}-\beta_{i+1}W_{r}v_{i}$. Combining the above two relations with \cref{GKB1}, we conclude that the iterative process for generating $u_i$ and $v_i$ is the standard GKB process of $AW_r$ with starting vector $\obs$ between the two finite dimensional Hilbert spaces $(\mathbb{R}^{r},\langle\cdot,\cdot\rangle_{2})$ and $(\mathbb{R}^{m},\langle\cdot,\cdot\rangle_{2})$. Therefore, $\{u_i\}_{i=1}^{k+1}$ and $\{v_i\}_{i=1}^{k+1}$ are two 2-orthonormal bases of the Krylov subspaces 
	\begin{align*}
		& \mathcal{K}_{k+1}(AW_{r}(AW_{r})^{\top}, b) = \mathrm{span}\{(AW_{r}W_{r}^{\top}A^{\top})^{i}b\}_{i=0}^{k},  \\
		& \mathcal{K}_{k+1}((AW_{r})^{\top}AW_{r}, (AW_{r})^{\top}b) = \mathrm{span}\{(W_{r}^{\top}A^{\top}AW_{r})^{i}W_{r}^{\top}A^{\top}b\}_{i=0}^{k},
	\end{align*}
	respectively; see e.g. \cite[\S 10.4]{Golub2013}. Then, $W_{r}W_{r}^{\top}= C_{rkhs}^{\dag}$ implies \cref{space_u}. Also, $\{z_i\}_{i=1}^{k+1}=\{W_{r}v_i\}_{i=1}^{k+1}$ is a $ C_{rkhs}$-orthonormal basis of $W_{r}\mathcal{K}_{k+1}((AW_{r})^{\top}AW_{r}, (AW_{r})^{\top}b)$ since $W_r$ is $ C_{rkhs}$-orthonormal. 
	
	Finally, $\{z_i\}_{i=1}^{k+1}$ are $C_{rkhs}$ orthogonal by construction, and by using the relation
	\begin{align*}
		W_{r}(W_{r}^{\top}A^{\top}AW_{r})^{i}W_{r}^{\top}A^{\top}b
		= (W_{r}W_{r}^{\top}A^{\top}A)^{i}W_{r}W_{r}^{\top}A^{\top}b
		= ( C_{rkhs}^{\dag}A^{\top}A)^{i} C_{rkhs}^{\dag}A^{\top}b, 
	\end{align*}
we get that $\{z_i\}_{i=1}^{k+1}$ in the RKHS-restricted Krylov subspace in \cref{space_z}. 
\end{proof}

\begin{proposition}[gGKB termination number]
\label{prop:gGKB_iter_num}
 Suppose the gGKB in {\rm\Cref{alg:gGKB}} terminates at step $k_{t} = \max_{k\geq 1}\{\alpha_{k}\beta_{k}>0\}$. Let the distinct nonzero eigenvalues of $AC_{rkhs}^{\dag}A^{\top}$ be $\mu_{1}>\cdots\mu_{s}>0$ with multiplicities $m_1,\dots,m_s$, and the corresponding eigenspaces are $\mathcal{G}_1,\cdots,\mathcal{G}_s$. Then, $k_{t}=q$, where $q$ is the number of nonzero elements in $\{P_{\mathcal{G}_1}b,\dots,P_{\mathcal{G}_s}b\}$, and 
\begin{equation}\label{eq:dim_Krylov}
	q= \mathrm{dim}(\mathcal{K}_{\infty}(A C_{rkhs}^{\dag}A^{\top}, b) ) = \mathrm{dim}( \mathcal{K}_{\infty}( C_{rkhs}^{\dag}A^{\top}A,  C_{rkhs}^{\dag}A^{\top}b) ), 
\end{equation}
where $K_\infty(M,v) = \mathrm{span}\{M^iv\}_{i=0}^\infty$ denotes the Krylov subspace of $\{M,v\}$. Moreover, $\sum_{i=1}^{s}m_i=r$ with $r$ being the rank of $A$ and $k_t=q\leq r$.   
 \end{proposition}
\begin{proof} 
First, we prove \cref{eq:dim_Krylov} with $q$ being the number of nonzero elements in $\{P_{\mathcal{G}_1}b,\dots,P_{\mathcal{G}_s}b\}$. 
Let $g_j=P_{\mathcal{G}_j}b/\|P_{\mathcal{G}_j}b\|_{2}$ for $1\leq j\leq s$, and let $g_1,\dots,g_q\neq \mathbf{0}$ without loss of generality. Note that $\{g_j\}_{j=1}^q$ are orthonormal. Let  $G_j$ be a matrix with orthonormal columns that span $\mathcal{G}_j$. Note that $P_{\mathcal{G}_j}=G_{j}G_{j}^{\top}$.  
By the eigenvalue decomposition $AC_{rkhs}^{\dag}A^{\top}=\sum_{j=1}^{s} \mu_{j}G_{j}G_{j}^{\top}$, we have
\begin{equation}\label{eq:wi}
    w_i := (AC_{rkhs}^{\dag}A^{\top})^{i-1}b = \sum_{j=1}^{s}\mu_{j}^{i-1}G_{j}G_{j}^{\top}b= \sum_{j=1}^{q}\mu_{j}^{i-1}\|P_{\mathcal{G}_j}b\|_2 g_{j}. 
\end{equation}
 % Since $P_{\mathcal{G}_j}=G_{j}G_{j}^{\top}$, by letting $g_j=P_{\mathcal{G}_j}b/|P_{\mathcal{G}_j}b\|_{2}$ and supposing $g_1,\dots,g_q\neq \mathbf{0}$ without loss of generality, we have $w_i=\sum_{j=1}^{q}\mu_{j}^{i-1}g_{j}g_{j}^{\top}b$, where $\{g_{j}\}_{j=1}^{q}$ are mutually orthogonal. 
 Hence, $ \mathrm{rank}\{w_i\}_{i=1}^{\infty}\leq q$. On the other hand, for $1\leq k\leq q$, setting $\bar{w}_i=\|P_{\mathcal{G}_j}b\|_2 g_{j}$, we have $(w_1\dots,w_k)=(\bar{w}_1,\dots,\bar{w}_q)T_k$ with 
\begin{equation*}
    T_k = \begin{pmatrix}
        1 & \mu_{1} & \cdots & \mu_{1}^{k-1} \\
        1 & \mu_{2} & \cdots & \mu_{2}^{k-1} \\
        \vdots & \vdots & \cdots & \vdots \\
        1 & \mu_{q} & \cdots & \mu_{q}^{k-1}
          \end{pmatrix} 
    	= \begin{pmatrix}
 	        T_{k1} \\
 	        T_{k2}
	 	  \end{pmatrix},
\end{equation*}
where $T_{k1}\in \R^{k\times k}$ consists of the first $k$ rows of $T_k$, and $T_{k2}\in \R^{k,q-k}$ consists of the rest rows.   
Note that  $T_{k1}$ is a Vandermonde matrix and it is nonsingular since $\mu_{i}\neq\mu_{j}$ for $1\leq i\neq j\leq k$. Then, $T_k$ has full column rank, thereby the rank of $\{w_i\}_{i=1}^{k}$ is $k$ for $1\leq k\leq q$, and  $ \mathrm{rank}\{w_i\}_{i=1}^{\infty} \geq \mathrm{rank}\{w_i\}_{i=1}^{q} = q$. Therefore, we have $\mathrm{dim}(\mathcal{K}_{\infty}(A C_{rkhs}^{\dag}A^{\top}, b) ) =  \mathrm{rank}\{w_i\}_{i=1}^{\infty}=q$. 

Also, we have $\mathrm{dim}( \mathcal{K}_{\infty}( C_{rkhs}^{\dag}A^{\top}A,  C_{rkhs}^{\dag}A^{\top}b) )   = \mathrm{rank}\{C_{rkhs}^{\dag}A^{\top}w_i\}_{i=1}^{\infty}  =q$, where the first equality follows from 
\begin{equation}\label{eq:Krylov_wi}
    (C_{rkhs}^{\dag}A^{\top}A)^{i-1} C_{rkhs}^{\dag}A^{\top}b= C_{rkhs}^{\dag}A^{\top}(A C_{rkhs}^{\dag}A^{\top})^{i-1} b = C_{rkhs}^{\dag}A^{\top}w_i, 
\end{equation} 
and the second equality follows from $\mathrm{rank}\{C_{rkhs}^{\dag}A^{\top}w_i\}_{i=1}^{\infty} =  \mathrm{rank}\{w_i\}_{i=1}^{\infty} =q$ since $C_{rkhs}^{\dag}A^{\top}$ is non-singular on $\mathrm{span}\{w_i\}_{i=1}^{\infty}$, which is a subset of $\mathrm{span}\{\mathcal{G}_l\}_{l=1}^s$ by \cref{eq:wi}. In fact, $C_{rkhs}^{\dag}A^{\top}$ is non-singular on $\mathrm{span}\{\mathcal{G}_l\}_{l=1}^s$ because $\{\mathcal{G}_l\}$ are eigenspaces of $AC_{rkhs}^{\dag}A^{\top}$ corresponding to the positive eiengvalues. 

Furthermore, \cref{eq:Krylov_wi} and the non-degeneracy of  $C_{rkhs}^{\dag}A^{\top}$ on $\mathrm{span}\{w_i\}_{i=1}^{\infty}$ imply that
\begin{equation}\label{eq:Krylov_q}
\begin{aligned}
\mathrm{dim} (\mathcal{K}_{q}( C_{rkhs}^{\dag}A^{\top}A,  C_{rkhs}^{\dag}A^{\top}b) ) = & \mathrm{rank}(\{ C_{rkhs}^{\dag}A^{\top}A)^{i-1} C_{rkhs}^{\dag}A^{\top}b \}_{i=0}^{q-1} \\
 = & \mathrm{rank}(\{	C_{rkhs}^{\dag}A^{\top}w_i \}_{i=1}^{q} = \mathrm{rank}\{w_i\}_{i=1}^{q} =q. 
\end{aligned} 	
\end{equation}
That is, the vectors $\{ C_{rkhs}^{\dag}A^{\top}A)^{i-1} C_{rkhs}^{\dag}A^{\top}b \}_{i=0}^{q-1}$ are linearly independent. 

Next, we prove that $k_{t}=q$.
Clearly, $k_{t}\leq q$ since by \Cref{prop:onb-GKB} $\{z_i\}_{i=1}^{k_t}$ are orthogonal and they are in $\mathcal{K}_{k_t}( C_{rkhs}^{\dag}A^{\top}A,  C_{rkhs}^{\dag}A^{\top}b) \subset \mathcal{K}_{\infty}( C_{rkhs}^{\dag}A^{\top}A,  C_{rkhs}^{\dag}A^{\top}b)$, whose dimension is $q$. On the other hand, we show next that if $k_t<q$, there will be a contradiction; hence, we must have $k_t=q$. In fact, \cref{GKB2} and \cref{GKB4m} imply that, for each $1\leq i\leq k_t$, 
\begin{align*}
    C_{rkhs}^{\dag}A^{\top}Az_{i} 
    &= \alpha_{i}C_{rkhs}^{\dag}A^{\top}u_{i} + \beta_{i+1}C_{rkhs}^{\dag}A^{\top}u_{i+1} \\
    &= \alpha_{i}(\alpha_{i}z_{i}+\beta_{i}z_{i+1}) + \beta_{i+1}(\alpha_{i+1}z_{i+1}+\beta_{i+1}z_{i}),
\end{align*}
which leads to 
\begin{equation*}
    \alpha_{i+1}\beta_{i+1}z_{i+1} = C_{rkhs}^{\dag}A^{\top}Az_{i}-(\alpha_{i}^{2}+\beta_{i+1}^{2})z_{i}-\alpha_{i}\beta_{i}z_{i-1}.
\end{equation*}
Note that $\alpha_{1}\beta_{1}z_{1}=C_{rkhs}^{\dag}A^{\top}\beta_{1}u_{1}=C_{rkhs}^{\dag}A^{\top}b $.
Combining the above two relations and using  $\alpha_k\beta_k>0$ for all $k\leq k_t$, it follows that $z_{k} \in \mathrm{span}\{( C_{rkhs}^{\dag}A^{\top}A)^{i} C_{rkhs}^{\dag}A^{\top}b\}_{i=0}^{k}$ for all $k\leq k_t$. Hence, recursively applying $z_i = \frac{1}{\alpha_i\beta_i}  C_{rkhs}^{\dag}A^{\top}Az_{i-1} - \frac{1}{\alpha_i\beta_i} (\alpha_{i-1}^2+\beta_i^2) z_{i-1}- \frac{\alpha_{i-1}\beta_{i-1}}{\alpha_i\beta_i}z_{i-2}$ for all $2\leq i\leq k_t$, we can write 
\[\alpha_{k_t+1}\beta_{k_t+1}z_{k_t+1}=\sum_{i=0}^{k_t}\xi_i( C_{rkhs}^{\dag}A^{\top}A)^{i} C_{rkhs}^{\dag}A^{\top}b,\]
with $\xi_i\in \R$ and in particular, $\xi_{k_t}=1/\Pi_{i=1}^{k_t}\alpha_i\beta_i\neq 0$. Now $\alpha_{k_t+1}\beta_{k_t+1}=0$ implies that $\{( C_{rkhs}^{\dag}A^{\top}A)^{i} C_{rkhs}^{\dag}A^{\top}b\}_{i=0}^{k_t}$ are linearly dependent, contradicting with the fact that they are linearly independent as suggested by \cref{eq:Krylov_q}. Therefore, we have $q= k_t$.

  Lastly, to prove that $\sum_{i=1}^{s}m_i=r$, it suffices to show that $\mathrm{rank}(AC_{rkhs}^{\dag}A^{\top})=r$ since the eigenvalues of $AC_{rkhs}^{\dag}A^{\top}$ are nonnegative. To see that its rank is $r$, following the proof of \Cref{prop:onb-GKB}, we write $AC_{rkhs}^{\dag}A^{\top}=AW_{r}W_{r}^{\top}A^{\top}$. Since $AW_{r}=AV_{r}\Lambda_{r}^{1/2}$, we only need to prove that $AV_{r}$ has full column rank. Suppose $AV_{r}y=\mathbf{0}$ with $y\in\mathbb{R}^{r}$. By \Cref{thm:RKHS_norm} it follows $A^{\top}AV_{r}y=\mathbf{0}\Leftrightarrow BV_{r}\Lambda_{r}y=\mathbf{0}\Leftrightarrow y=\mathbf{0}$. Thus, $\mathrm{rank}(AC_{rkhs}^{\dag}A^{\top})=\mathrm{rank}(AW_{r})=r$.
\end{proof}

\subsection{Uniqueness of solutions in the iterations}

\begin{proposition}\label{naive}
If gGKB terminates at step $k_t$ in \cref{eq:gGK_niter}, the iterative solution $x_{k_t}$ is the unique least squares solution to $\min_{x\in\mathcal{R}( C_{rkhs})}\|Ax-b\|_2$.
\end{proposition}
\begin{proof}
Following the proof of \Cref{prop:onb-GKB}, the solution to $\min_{x\in\mathcal{R}( C_{rkhs})}\|Ax-b\|_2$ is $x_{\star}=W_{r}y_{\star}$ with $y_{\star}=\argmin_{y}\|AW_{r}y-b\|_{2}$. Since $AW_{r}$ has full column rank, it follows that $y_{\star}$ is the unique solution to $W_{r}^{\top}A^{\top}(AW_{r}y-b)=\mathbf{0}$. Note that $\mathcal{R}(W_{r})=\mathcal{R}( C_{rkhs})$. Thus, $x_{\star}$ is the solution to $\min_{x\in\mathcal{R}( C_{rkhs})}\|Ax-b\|_2$ if and only if $P_{\mathcal{R}( C_{rkhs})}A^{\top}(Ax_{\star}-b)=\mathbf{0}$.

Now we only need to prove $P_{\mathcal{R}( C_{rkhs})}A^{\top}(Ax_{q}-b)=\mathbf{0}$ since $k_t= q$ by \Cref{prop:gGKB_iter_num}. Using the property $P_{\mathcal{R}( C_{rkhs})}=C_{rkhs}C_{rkhs}^{\dag}$, we get from \cref{GKB8} 
\begin{equation*}
    P_{\mathcal{R}( C_{rkhs})}A^{\top}U_{k+1} = C_{rkhs}(Z_kB_{k}^{T}+\alpha_{k+1}z_{k+1}e_{k+1}^\top ).
\end{equation*}
Combining the above relation with \cref{eq:min-Sk}, we have
\begin{align*}
    P_{\mathcal{R}( C_{rkhs})}A^{\top}(Ax_{q}-b) 
    &= C_{rkhs}(Z_qB_{q}^{T}+\alpha_{q+1}z_{q+1}e_{q+1}^\top )(B_{q}y_{q}-\beta_{1}e_{1}) \\
    &= C_{rkhs}[Z_{q}(B_{q}^{\top}B_{q}y_{q}- B_{q}^{\top}\beta_{1}e_{1}) + \alpha_{q+1}\beta_{q+1}z_{q+1}e_{q}^{\top}y_{q}] \\
    &= \alpha_{q+1}\beta_{q+1}C_{rkhs}z_{q+1}e_{q}^{\top}y_{q} = \boldsymbol{0},
\end{align*}
since $\alpha_{q+1}\beta_{q+1} = 0$ when gGKB terminates.
\end{proof}

This result shows the necessity for early stopping the iteration to avoid getting a naive solution. The next theorem shows the uniqueness of the solution in each iteration of the algorithm. 

\begin{theorem}[Uniquess of solution in each iteration]\label{thm:solu_uniqueness}
For each iteration with $k<k_t$, there exists a unique solution to \cref{subspace_regu}. Furthermore, there exists a unique solution to $ \min_{\vecphi\in\mathcal{S}_k}\|A\vecphi-\obs\|_{2}$. 
% (as long as $C_{rkhs}$ is non-singular on $\mathcal{X}_k$. \FL{Potential issues: $\|\vecphi\|_{C_{rkhs}}=\infty$ for all $\vecphi\in\mathcal{X}_k$ for some $\mathcal{X}_k$. })
\end{theorem}
\begin{proof}
	Let $W_k\in\mathbb{R}^{n\times k}$ that has orthonormal columns and spans $\mathcal{S}_k$. For any $\vecphi\in\mathcal{S}_k$, there is a unique $ y \in\mathbb{R}^k$ such that $\vecphi=W_k y $. 
Then the solution to \cref{subspace_regu} should be $\vecphi_k=W_k y _k$, where $ y _k$ is the solution to 
	\[\min_{ y \in\mathcal{Y}_k}\|W_k y \|_{ C_{rkhs}}, \ \ 
	\mathcal{Y}_k=\{ y : \min_{ y \in\mathbb{R}^{k}}\|AW_k y -b\|_2\} .\]
	By \cite[Theorem 2.1]{Elden1982}, it has a unique solution $ y _k$ iff $\mathcal{N}( C_{rkhs}^{1/2}W_k)\bigcap  \mathcal{N}(AW_k)=\{\boldsymbol{0}\}$. 
	
    Now we prove $\mathcal{N}( C_{rkhs}^{1/2}W_k)=\{\boldsymbol{0}\}$. To this end, suppose $ y \in\mathcal{N}( C_{rkhs}^{1/2}W_k)$ and $\vecphi=W_k y $. Then $\vecphi\in\mathcal{S}_k\bigcap \mathcal{N}( C_{rkhs}^{1/2})$. Since $\mathcal{N}( C_{rkhs}^{1/2})=\mathcal{N}( C_{rkhs})=\mathcal{R}( C_{rkhs})^{\perp}$, % where we have used that $ C_{rkhs}$ is symmetric, 
    we get $\vecphi\in\mathcal{S}_k\bigcap\mathcal{R}( C_{rkhs})^{\perp}\subseteq \mathcal{R}( C_{rkhs})\bigcap \mathcal{R}( C_{rkhs})^{\perp}=\{\boldsymbol{0}\}$. Therefore, $\vecphi=W_k y =\boldsymbol{0}$, which leads to $ y =\boldsymbol{0}$.

To prove the uniqueness of a solution to $ \min_{\vecphi\in\mathcal{S}_k}\|A\vecphi-\obs\|_{2}$, suppose that there are two minimizers, $x_1 \neq x_2 \in \mathcal{S}_k$, and we prove that they must be the same. Let $x_*=x_1-x_2 $ and we have $A^\top A x_* =0 $ since the minimizer of $\calE(x) = \|A\vecphi-\obs\|_{2} $ must satisfy $0 = \frac{1}{2}\nabla \calE(x) = A^\top A x- A^\top b$. That is, $x_*\in \mathcal{N}(A^\top A)$. 

On the other hand, since $ x_*\in \mathcal{S}_k \subset \mathcal{R}(C_{rkhs})$ by \Cref{prop:z_in_RKHS}, and note that $C_{rkhs}=B(A^{\top}A)^{\dag}B=BV_{r}\Lambda_{r}^{-1}V_{r}^{\top}B$, we have $B^{-1}x_*\in \mathcal{R}((A^{\top}A)^{\dag}B) \subset \mathcal{N}(A^\top A)^\perp$.  

Combining the above two, we have $\innerp{x_*,B^{-1}x_*}=0$. But $B$ is a symmetric positive definite matrix, so we must have $x_*=0$.  
\end{proof}

%%%%%%%%%%%%%%%%%%%%%%%%%%%%%%
%%%%%%%%  ====== section =======  %%%%%%%%
%%%%%%%%%%%%%%%%%%%%%%%%%%%%%%
\section{Numerical Examples}\label{sec:num}

\subsection{The Fredholm equation of the first kind}\label{sec:FIE_num}
We first examine the iDARR in solving the discrete Fredholm integral equation of the first kind.  The tests cover two distinct types of spectral decay: exponential and polynomial.  The latter is well-known to be challenging, often occurring in applications such as image deblurring. Additionally, we investigate two scenarios concerning whether the true solution is inside or outside the function space of identifiability (FSOI).

% We compare iDARR with the widely used iterative regularizers based on $l^2$- and $L^2$-norms. They differ primarily in their regularization norms. Also, we compare them with the direct regularizers that utilize these three norms and are based on matrix decomposition.  

\noindent\textbf{Three norms in iterative and direct methods.} We compare the $l^2$, $L^2$, and DA-RKHS norms in iterative and direct methods. The direct methods are based on matrix decomposition. These regularizers are listed in \Cref{tab:regularizers}.  
\begin{table}[ht] 
	\begin{center} 
		\caption{ Three regularization norms in iterative and direct methods.   
		 } \label{tab:regularizers}
		\begin{tabular}{ c   c | c c }		\toprule % \hline
                Norms            & $\|x\|_*^2$ &  Iterative  & Direct  \\  \midrule
	         $l^2$   &   $x^\top I_n x $  & IR-l2   &l2 \\ 
	         $L^2$    &   $x^\top Bx $  &  IR-L2 & L2\\   
	       DA-RKHS     &  $ x^\top C_{rkhs} x$  & iDARR   & DARTR \\
	 			\bottomrule	  
		\end{tabular}  \vspace{-4mm}
	\end{center}
\end{table} 

The iterative methods differ primarily in their regularization norms. For the with $l^2$ norm, we use the LSQR method in the \textsf{IR TOOLS} package \cite{gazzola2019ir}, and we stop the iteration when $\alpha_i$ or $\beta_i$ becomes negligible to maintain stability. For the $L^2$ norm, we use gGKB to construct solution subspaces by replacing $C_{rkhs}$ by the basis matrix $B$; note that this method is equivalent to the LSQR method using $L=\sqrt{B}$ as a preconditioner in the \textsf{IR TOOLS} package. 

The direct methods are Tikhonov regularizers using the L-curve method \cite{hansen_LcurveIts_a}.

\noindent\textbf{Numerical settings.}  We consider the problem of recovering the input signal $\phi$ in a discretization of Fredholm integral equation in \cref{eq:FIE} with $s\in [a,b]$ and $t\in [c,d]$. 
The data are discrete noisy observations $b=(y(t_1),\ldots, y(t_m))\in \R^m$, where $ t_j =c+ j(d-c)/m$ for $0\leq j\leq m$.  The task is to estimate the coefficient vector $x=  (\phi(s_1), \ldots,\phi(s_n)) \in \R^n$ in a piecewise-constant function $\phi(s)=\sum_{i=1}^n \phi(s_i) \mathbf{1}_{[s_{i-1},s_i]}(s)$, where $\calS:= \{s_i\}_{i=1}^n\subset [a,b]$ with $s_i= a+i\delta $, $ \delta=(b-a)/n$. 
 We obtain the linear system \cref{eq:Ab} with $A(j,i) = K(t_j,s_i)\delta$ by a Riemann sum approximation of the integral.  We set $(a,b,c,d)=(1,5,0,5)$,  $m=500$, and take $n =100$ except when testing the computational time with a sequence of large values for $n$.

The $\R^m$-valued noise $\bw$ is Gaussian $N(0,\sigma^2\Delta t I_m)$. We set the standard deviation of the noise to be $\sigma= \|Ax  \|  \times nsr$, where $nsr$ is the noise-to-signal ratio, and we test our methods with $nsr = \{0.0625,0.125,0.25,0.5,1\}$.

We consider two integral kernels 
\begin{equation}\label{eq:kernels}
\text{(a) }\, K(t,s):= s^{-2}e^{- s t}; \quad \text{(b) }\, K(t,s):= s^{-1} |\sin(st +1)|. 
\end{equation}
which lead to exponential and polynomial decaying spectra, respectively, as shown in Figure \ref{fig:kernel_spectrum}. 
The first kernel arises from magnetic resonance relaxometry \cite{Bi2022span-of-regular} with $\phi$ being the distribution of transverse nuclear relaxation times.

\begin{figure}[!htb] \centering
\subfigure[Exponential decaying spectrum]{ \includegraphics[width=0.36\textwidth]{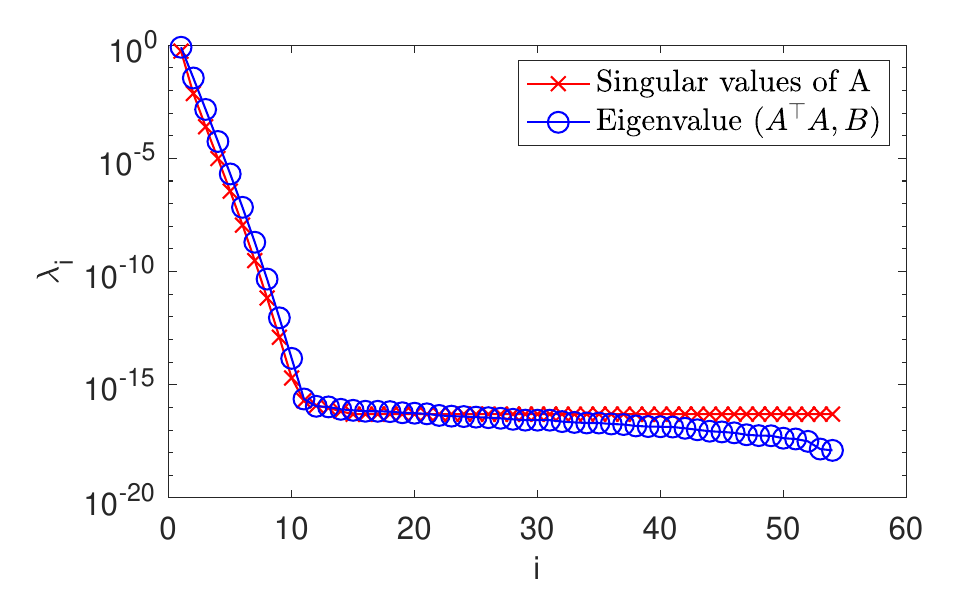} }
\subfigure[Polynomial decaying spectrum]{\includegraphics[width=0.355\textwidth]{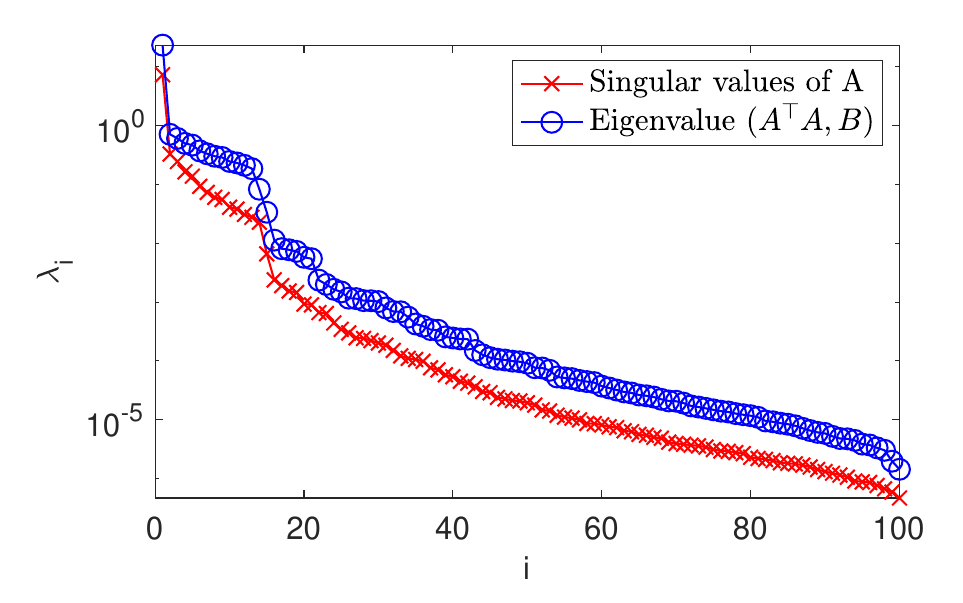} }\vspace{-3mm}
\caption{Singular values of $A$ and generalized eigenvalues of $(A^\top A, B)$ for kernels in \cref{eq:kernels}. % (a) $K(t,s):= s^{-2}e^{- s t}$; (b) $K(t,s):= s^{-1} |\sin(st +1)|$.
}
\label{fig:kernel_spectrum}
\end{figure}

\noindent\textbf{The DA-RKHS.} By its definition in \cref{eq:exp_measure}, the exploration measure is  $ \rho(s_i)=\frac{1}{\gamma} \sum_{j=1}^m |K(t_j,s_i)|\delta$ for  $i=1,\dots,n$ with $\gamma$ being the normalizing constant. In other words, it is the normalized column sum of the absolute values of the matrix $A$. The discrete function space $L^2_\rho(\calS)$ is equivalent to $\R^n$ with weight $\rho=(\rho(s_1),\ldots, \rho(s_n))$, and its norm is $\innerp{x,x}_{L^2_\rho}:=x^\top \mathrm{diag}(\rho(s_i))x$ for all $x \in \R^n$.
The basis matrix for the Cartesian basis of $\R^n$ is 
% \begin{equation}\label{eq:basisMat}
 $B = \mathrm{diag}(\rho(s_i))$, 
 %\end{equation}
 which is also the basis matrix for step functions in the Riemann sum discretization. The DA-RKHS in this discrete setting is $(\mathcal{N}(A^\top A) )^\perp, \innerp{\cdot,\cdot}_{C_{rkhs}})$.

\noindent\textbf{Settings for comparisons.} The comparison consists of two scenarios regarding whether the true solution is inside or outside of the FSOI: (i) the true solution is the second eigenvector of $\LGbar$; thus it is inside the FSOI; and (ii) the true solution is $\phi(x) = x^2$, which has significant components outside the FSOI. 

For each scenario, we conduct 100 independent simulations, with each simulation comprising five datasets at varying noise levels. The results are presented by a box plot, which illustrates the median, the lower and upper quartiles, any outliers, and the range of values excluding outliers. The key indicator of a regularizer's effectiveness is its ability to produce accurate estimators whose errors decay consistently as the noise level decreases. Since exploring the decay rate in the small noise limit is not the focus of this study, we direct readers to \cite{LuOu23,LangLu23sna} for initial insights into how this rate is influenced by factors such as spectral decay, the smoothness of the true solution, and the choice of regularization strength.

\begin{figure}[htb]
\centering
% \text{True function in FSOI} \hspace{30mm} \text{True function outside FSOI} \\
% \includegraphics[width=0.45\textwidth]{Ex_inside_FSOI/insideFSOI_IRexp_estimator.pdf}
% \includegraphics[width=0.45\textwidth]{Ex_outside_FSOI/outsideFSOI_IRexp_estimator.pdf}
% \scriptsize{\hspace{12mm} \rotatebox{90}{\hspace{7mm}\textbf{Iterative methods} }}
% \includegraphics[width=0.22\textwidth,height=0.22\textwidth]{Ex_inside_FSOI/insideFSOI_IRexp_L2rho_error.pdf}
% \includegraphics[width=0.22\textwidth,height=0.22\textwidth]{Ex_inside_FSOI/insideFSOI_IRexp_loss.pdf}
% \includegraphics[width=0.22\textwidth,height=0.22\textwidth]{Ex_outside_FSOI/outsideFSOI_IRexp_L2rho_error.pdf}
% \textbf{\includegraphics[width=0.22\textwidth,height=0.22\textwidth]{Ex_outside_FSOI/outsideFSOI_IRexp_loss.pdf}}

% \scriptsize{\rotatebox{90}{\hspace{9mm}\textbf{Direct methods} }}
% \includegraphics[width=0.22\textwidth,height=0.22\textwidth]{Ex_inside_FSOI/insideFSOI_Lcurveexp_L2rho_error.pdf}
% \includegraphics[width=0.22\textwidth,height=0.22\textwidth]{Ex_inside_FSOI/insideFSOI_Lcurveexp_loss.pdf}
% \includegraphics[width=0.22\textwidth,height=0.22\textwidth]{Ex_outside_FSOI/outsideFSOI_Lcurveexp_L2rho_error.pdf}
% \includegraphics[width=0.22\textwidth,height=0.22\textwidth]{Ex_outside_FSOI/outsideFSOI_Lcurveexp_loss.pdf}
\includegraphics[width=0.98\textwidth]{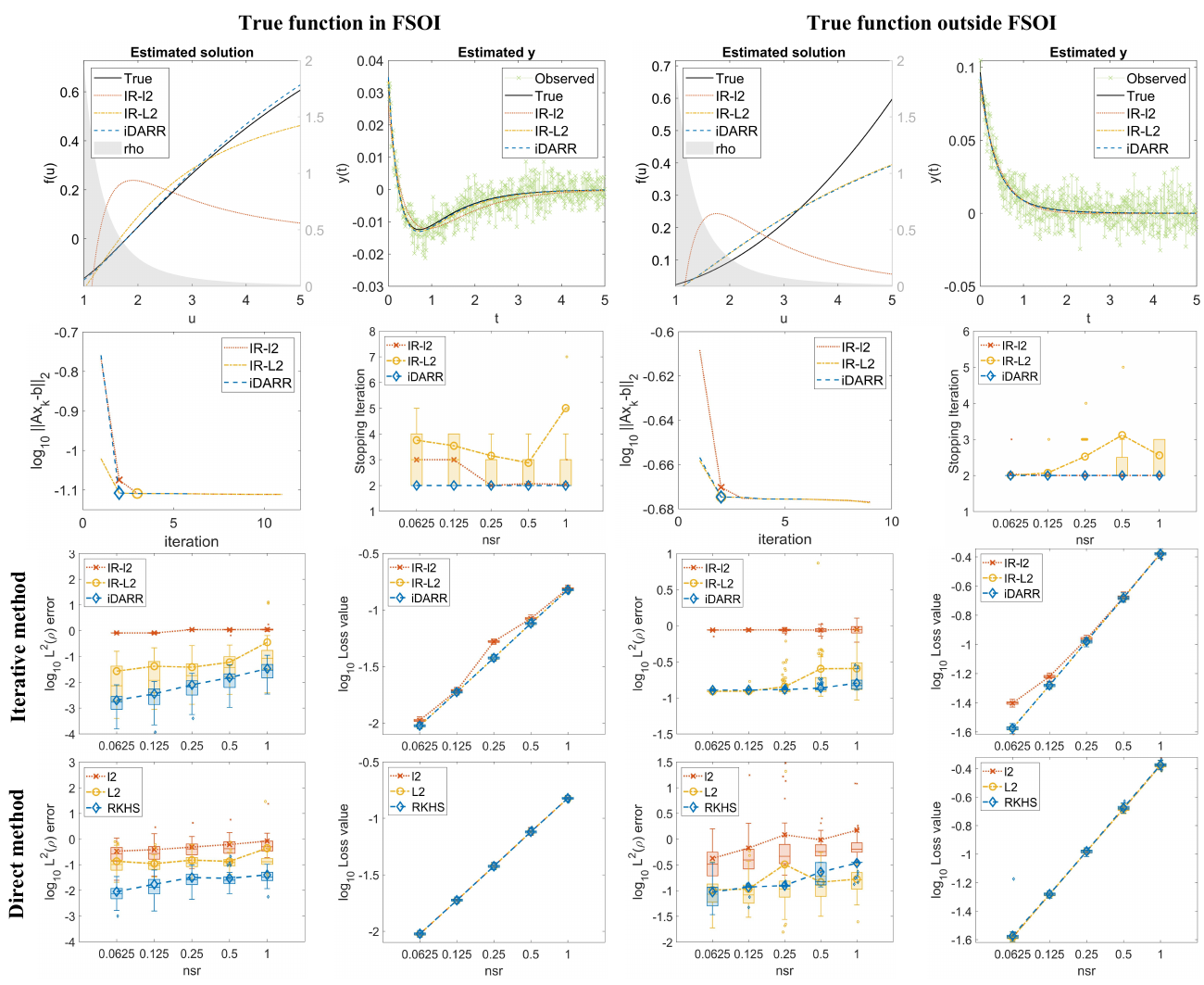}\vspace{-3mm}
\caption{ Results in the case of exponentially decaying spectrum. 
Top-row: typical estimators of IR-l2, IR-L2, and iDARR when $nsr=0.5$ and their denoising of the output signal. 
The 2nd-top row: the residual $\|Ax_k-b\|_2$ as iteration number $k$ increases in one realization when $nsr=0.5$, as well as the boxplots of the stopping iteration numbers in the 100 simulations. 
The lower two rows: boxplots of the estimators' $L^2_\rho$ errors and loss function values in the 100 simulations. 
}
\label{fig:estimator_exp}\vspace{-3mm}
\end{figure}

\noindent\textbf{Results.}  We report the results separately according to the spectral decay. 
 
\textbf{(i) Exponential decaying spectrum.} %The kernel $K(t,s):= s^{-2}e^{- s t}$ \\
\Cref{fig:estimator_exp}'s top row shows typical estimators of IR-l2, IR-L2, and iDARR and their de-noising of the output signal when $nsr=0.5$. When the true solution is inside the FSOI, the iDARR significantly outperforms the other two in producing a more accurate estimator. However, both IR-L2 and IR-l2 can denoise the data accurately, even though their estimators are largely biased. When the true solution is outside the FSOI, all the regularizers can not capture the true function accurately, but iDARR and IR-L2 clearly outperform the $l_2$ regularizer. Yet again, all these largely biased estimators can de-noise the data accurately. Thus, this inverse problem is severely ill-defined, and one must restrict the inverse to be in the FSOI.

The 2nd top row of \Cref{fig:estimator_exp} shows the decay of the residual $\|Ax_k-y\|_2$ as the iteration number increases, as well as the stopping iteration numbers of these regularizers in 100 simulations. The fast decaying residual suggests the need for early stopping, and all three regularizers indeed stop in a few steps. Notably, iDARR consistently stops early at the second iteration for different noise levels, outperforming the other two regularizers in stably detecting the stopping iteration. 

The effectiveness of the DA-RKHS regularization becomes particularly evident in the lower two rows of \Cref{fig:estimator_exp}, which depict the decaying errors and loss values as the noise-to-signal ratio ($nsr$) decreases in the 100 independent simulations. In both iterative and direct methods, the DA-RKHS norm demonstrates superior performance compared to the $l^2$ and $L^2$ norms, consistently delivering more accurate estimators that show a steady decrease in error alongside the noise level. Notably, the values of the corresponding loss functions are similar, underscoring the inherent ill-posedness of the inverse problem. Furthermore, iDARR marginally surpasses the direct method DARTR in producing more precise estimators, particularly when the true solution resides within the FSOI. The performance of iDARR suggests that its early stopping mechanism can reliably determine an optimal regularization level, achieving results that are slightly more refined than those obtained with DARTR using the L-curve method.

% \ifjournal \vspace{-4mm} \fi
\begin{figure}[!htb]
\centering
% \text{True function in FSOI} \hspace{30mm} \text{True function outside FSOI} \\
% \includegraphics[width=0.45\textwidth]{Ex_inside_FSOI/insideFSOI_IRpoly_estimator.pdf}
% \includegraphics[width=0.45\textwidth]{Ex_outside_FSOI/outsideFSOI_IRpoly_estimator.pdf}
% \scriptsize{\hspace{12mm} \rotatebox{90}{\hspace{7mm}\textbf{Iterative methods} }}
% \includegraphics[width=0.22\textwidth,height=0.22\textwidth]{Ex_inside_FSOI/insideFSOI_IRpoly_L2rho_error.pdf}
% \includegraphics[width=0.22\textwidth,height=0.22\textwidth]{Ex_inside_FSOI/insideFSOI_IRpoly_loss.pdf}
% \includegraphics[width=0.22\textwidth,height=0.22\textwidth]{Ex_outside_FSOI/outsideFSOI_IRpoly_L2rho_error.pdf}
% \textbf{\includegraphics[width=0.22\textwidth,height=0.22\textwidth]{Ex_outside_FSOI/outsideFSOI_IRpoly_loss.pdf}}

% \scriptsize{\rotatebox{90}{\hspace{9mm}\textbf{Direct methods} }}
% \includegraphics[width=0.22\textwidth,height=0.22\textwidth]{Ex_inside_FSOI/insideFSOI_Lcurvepoly_L2rho_error.pdf}
% \includegraphics[width=0.22\textwidth,height=0.22\textwidth]{Ex_inside_FSOI/insideFSOI_Lcurvepoly_loss.pdf}
% \includegraphics[width=0.22\textwidth,height=0.22\textwidth]{Ex_outside_FSOI/outsideFSOI_Lcurvepoly_L2rho_error.pdf}
% \includegraphics[width=0.22\textwidth,height=0.22\textwidth]{Ex_outside_FSOI/outsideFSOI_Lcurvepoly_loss.pdf}
\includegraphics[width=0.98\textwidth]{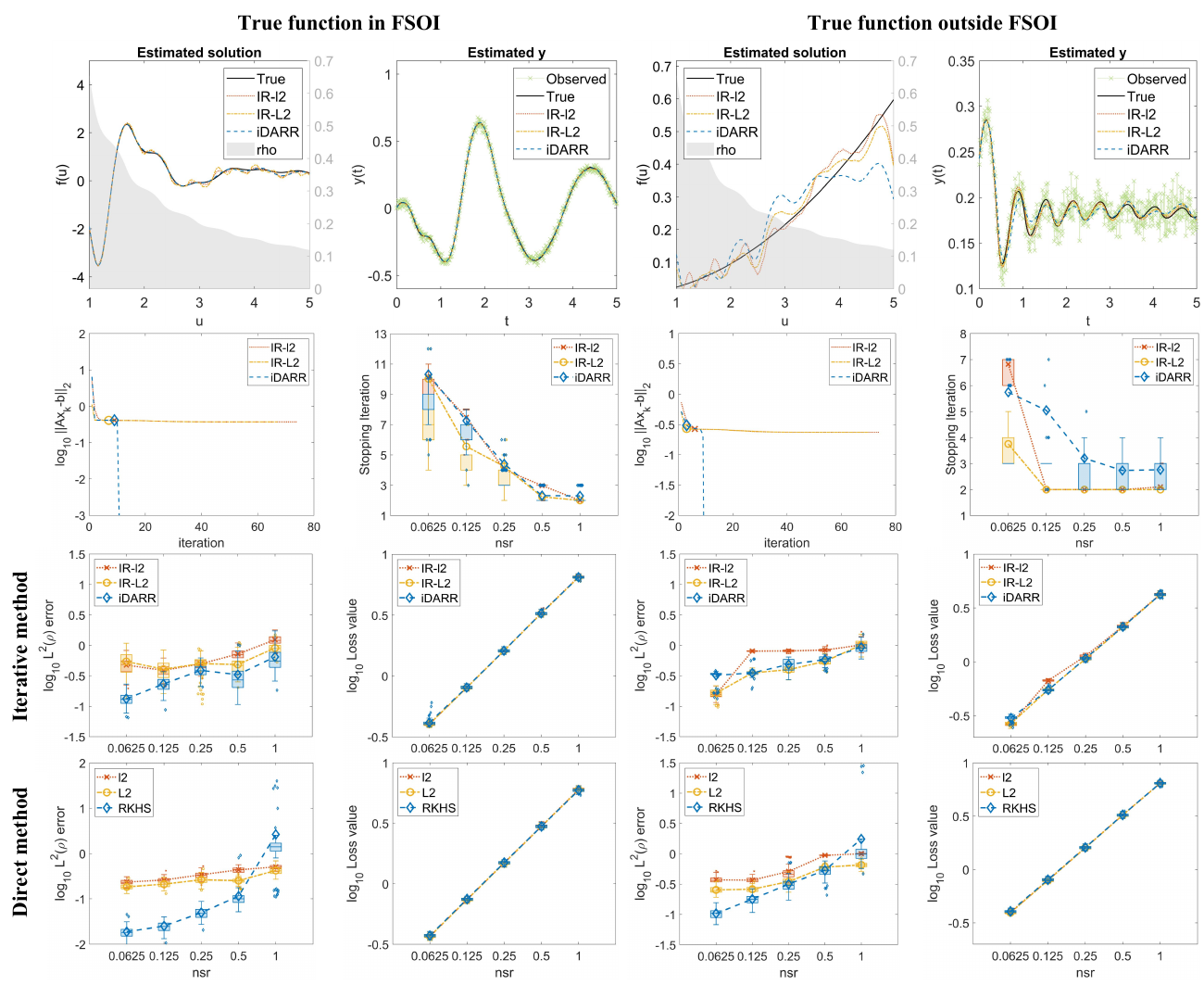}\vspace{-3mm}
\caption{Results in the case of polynomial decaying spectrum. 
 Top-row: typical estimators of IR-l2, IR-L2, and iDARR when $nsr=0.0625$ and their denoising of the output signal. 
 The 2nd-top row: the residual $\|Ax_k-b\|_2$ as iteration number $k$ increases in one realization when $nsr=0.0625$, as well as the box plots of the stopping iteration numbers the 100 simulations. 
 The lower two rows: box plots of the estimators' $L^2_\rho$ errors and loss function values in the 100 simulations. 
 }
\label{fig:estimator_poly}\vspace{-3mm}
\end{figure}
 \textbf{(ii) Polynomial decaying spectrum.} % $K(t,s):= s^{-1} |\sin(st +1)|$ 
% Figure \ref{fig:estimator_poly} shows typical regularized estimators of IR-l2, IR-L2, and iDARR and their recovery of the true output signal when $nsr=0.0625$. 
\Cref{fig:estimator_poly} illustrates again the superior performance of iDARR over IR-L2 and IR-l2 in the case of polynomial spectral decay. 
The 2nd top row shows that the slow spectral decay poses a notable challenge to the iterative methods, as the noise level affects their stopping iteration numbers. Also, they all stop early within approximately twelve steps, even though the true solution may lay in a subspace with a higher dimension. 
 
 The lower two rows show that iDARR remains effective. It continues to outperform the other two iterative regularizers when the true solution is in the FSOI, and it is marginally surpassed by IR-L2 and IR-l2 when the true solution is outside the FSOI. In both scenarios, the direct method DARTR outperforms all other methods, including iDARR, indicating DARTR is more effective in extracting information from the spectrum with slow decay.

\noindent\textbf{Computational Complexity.}
The iterative method iDARR is orders of magnitude faster than the direct method DARTR, especially when $n$ is large. \Cref{fig:time} shows their computation time as $n$ increases in $10$ independent simulations, and the results align with the complexity order illustrated in \Cref{subsec:complexity}.  

 % with $K(t,s):= s^{-2}e^{- s t}$, $nsr = 0.5$, The iterative method is orders of magnitude faster than both the hybrid method and the direct method. \FL{FL: to connect with the computational complexity section. }

 \begin{figure}[!htb]
\centering
\includegraphics[width=0.4\textwidth]{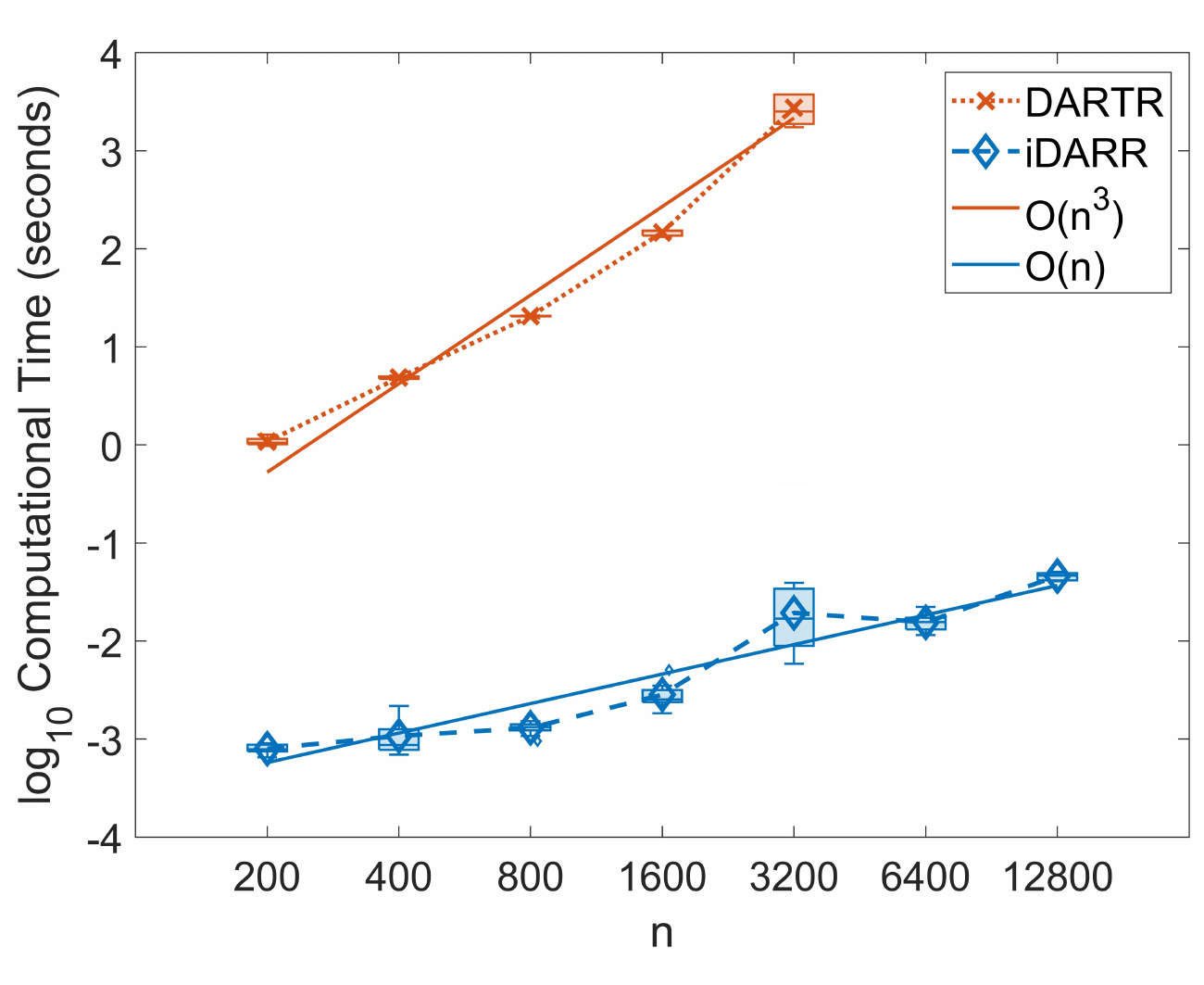} \vspace{-4mm}
\caption{Computational time in 10 simulations with $m=500$. % The iterative method iDARR has a complexity of $O(kmn)$ with $k$ being the number of iterations, often less than a dozen. In contrast, the direct method DARTR has a complexity of $O(mn^2+n^3)$,  computationally prohibitive when $n$ is large. 
}
\label{fig:time} \vspace{-4mm}
\end{figure}

In summary, iDARR outperforms IR-L2 and IR-l2 in yielding accurate estimators that consistently decay with the noise level. Its major advantage comes from the DA-RKHS norm that adaptively exploits the information in data and the model.

%% %%%%%%%%%
\subsection{Image Deblurring}
We further test iDARR in 2D image deblurring problems, where the task is to reconstruct images from blurred and noisy observations. The mathematical model of this problem can be expressed in the form of the first-kind Fredholm integral equation in \cref{eq:FIE} with $s,t \in \mathbb{R}^2$. 
% $g(t) = \int_{\Omega} k(t,s) f(s)ds+e(t)$, where $s,t \in \mathbb{R}^2$. 
The kernel $K(t,s)$ is a function that specifies how the points in the image are distorted, called the point spread function (PSF). We chose {\sf PRblurspeckle} from \cite{gazzola2019ir} as the blurring operator, which simulates spatially invariant blurring caused by atmospheric turbulence, and we use zero boundary conditions to construct the matrix $A$. For a true image with $N\times N$ pixels, the matrix $A\in\mathbb{R}^{N^2\times N^2}$ is a \texttt{psfMatrix} object. We consider two images with $256\times 256$ and $320\times 320$ pixels, respectively, and set the noise level to be $nsr=0.01$ for both images. The true images, their blurred noisy observations, and corresponding PSFs that define matrices $A$ are presented in \Cref{fig:image}.

\begin{figure}[htb]
\centering
\subfigure[True Image-1]{
\includegraphics[width=0.2\textwidth]{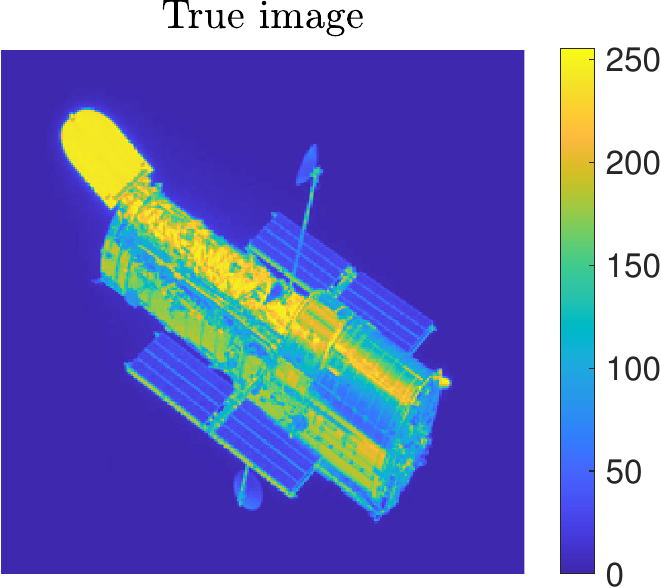}}\hspace{2mm}
\subfigure[Blurred Image-1]{
\includegraphics[width=0.2\textwidth]{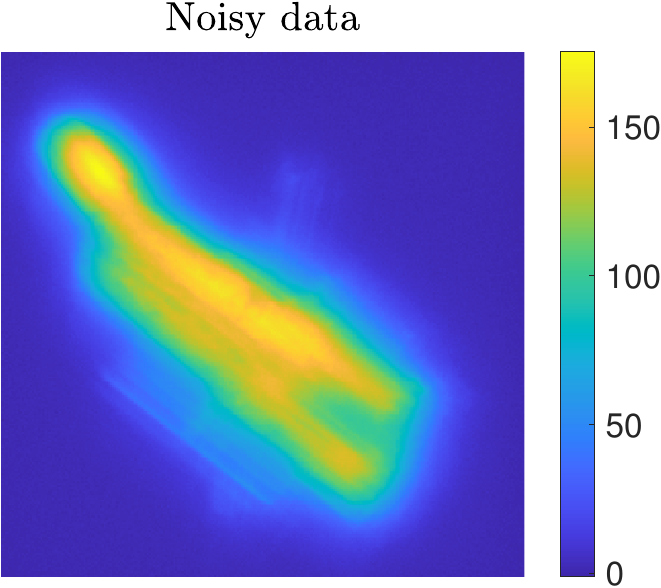}}\hspace{2mm}
\subfigure[PSF for $N=256$]{
\includegraphics[width=0.2\textwidth]{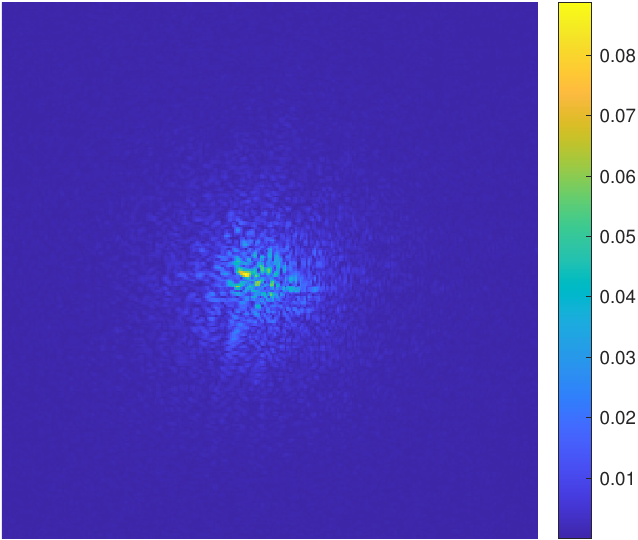}} \\
\vspace{-2mm}
\subfigure[True Image-2]{
\includegraphics[width=0.2\textwidth]{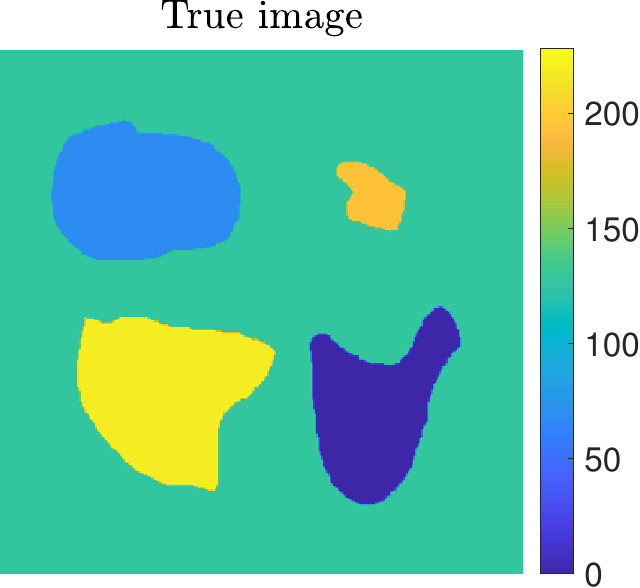}}\hspace{2mm}
\subfigure[Blurred Image-2]{
\includegraphics[width=0.2\textwidth]{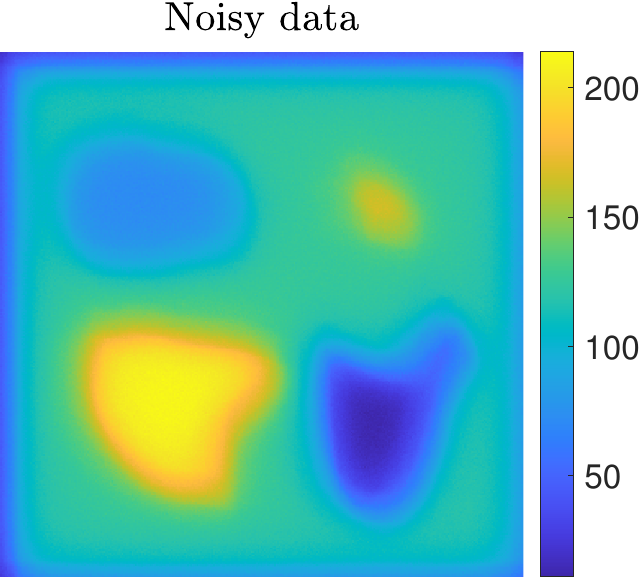}}\hspace{2mm}
\subfigure[PSF for $N=320$]{
\includegraphics[width=0.2\textwidth]{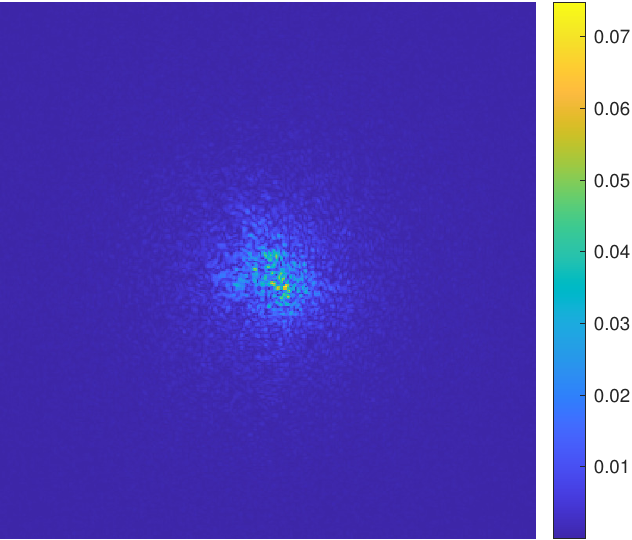}}
\vspace{-4mm}
\caption{The true images, noisy images blurred by {\sf PRblurspeckle}, and the corresponding PSFs.} 
\label{fig:image}
\end{figure}

% There are several test problems with various blurring operations included in the $IR\ TOOLS$ package that we can apply, e.g.
% \begin{itemize}
%     \item $PRblurspeckle$ simulates spatially invariant blurring caused by atmospheric turbulence.
%     \begin{equation}
%         A_{ij} = \exp \left\{-\frac{1}{2} \begin{pmatrix} i-k\\j-l\end{pmatrix}^T \begin{pmatrix} s_1^2 & \rho^2\\\rho^2 & s_2^2 \end{pmatrix}^{-1} \begin{pmatrix} i-k\\j-l\end{pmatrix} \right\}
%     \end{equation}
%     where the parameters $s_1$, $s_2$, and $\rho$ determine the width and the orientation of the PSF, which is centered at element $(k, l)$ in A \cite{hansen2006deblurring}. (Note that pixels of the blurred image near the boundary of the viewable region are affected by information outside the viewable region. Therefore, in constructing matrix A, one needs to incorporate boundary conditions to model how the image scene extends beyond the boundaries of the viewable region.)
% \end{itemize}

\ifjournal 
\begin{figure}[htb]
\centering
\subfigure{
\includegraphics[width=0.22\textwidth]{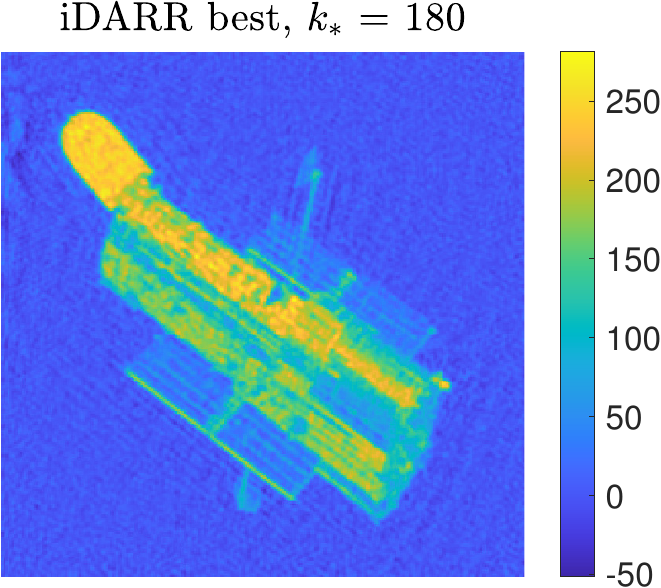}}\hspace{-7.5mm}
\subfigure{
\includegraphics[width=0.22\textwidth]{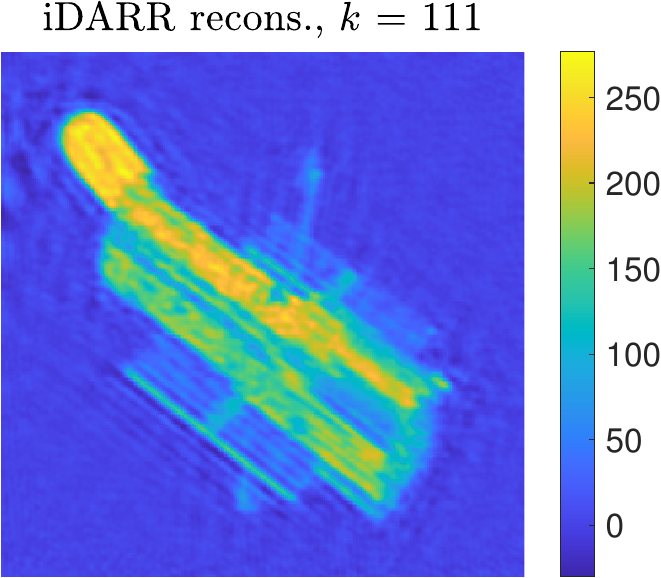}}\hspace{-7.5mm}
\subfigure{
\includegraphics[width=0.22\textwidth]{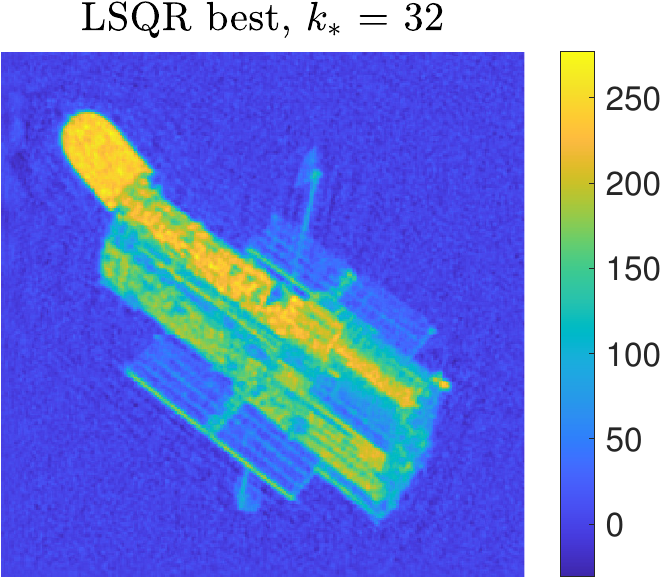}}\hspace{-7.5mm}
\subfigure{
\includegraphics[width=0.225\textwidth]{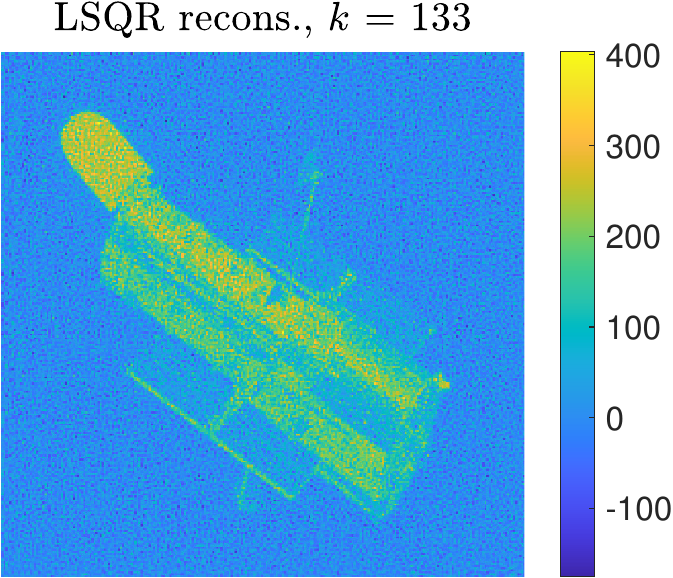}}\hspace{-7.8mm}
\subfigure{
\includegraphics[width=0.22\textwidth]{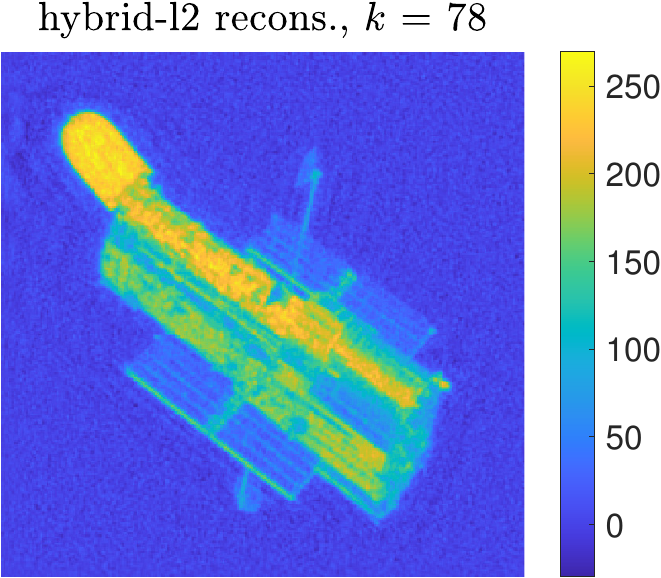}} \\
\vspace{-2mm}
\subfigure{
\includegraphics[width=0.22\textwidth]{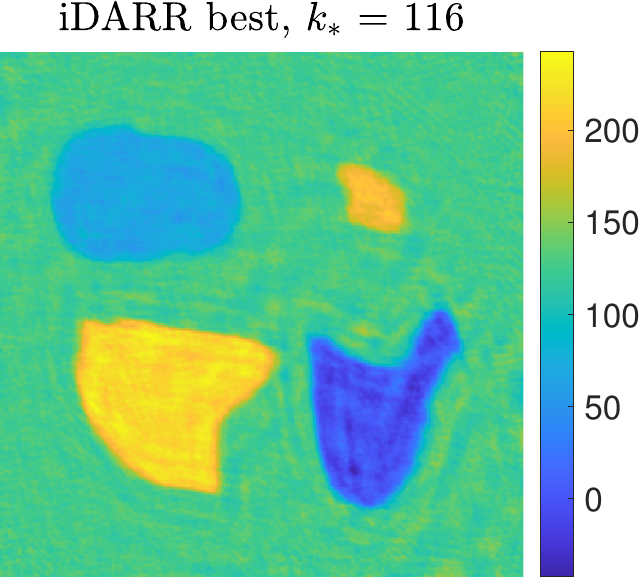}}\hspace{-7mm}
\subfigure{
\includegraphics[width=0.22\textwidth]{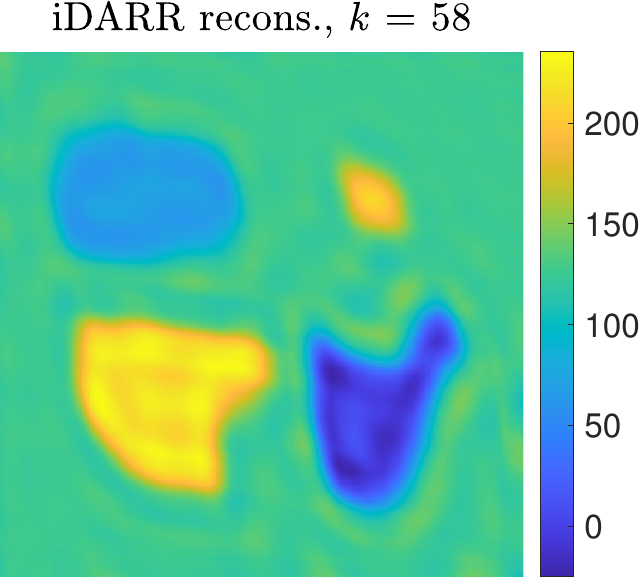}}\hspace{-7mm}
\subfigure{
\includegraphics[width=0.22\textwidth]{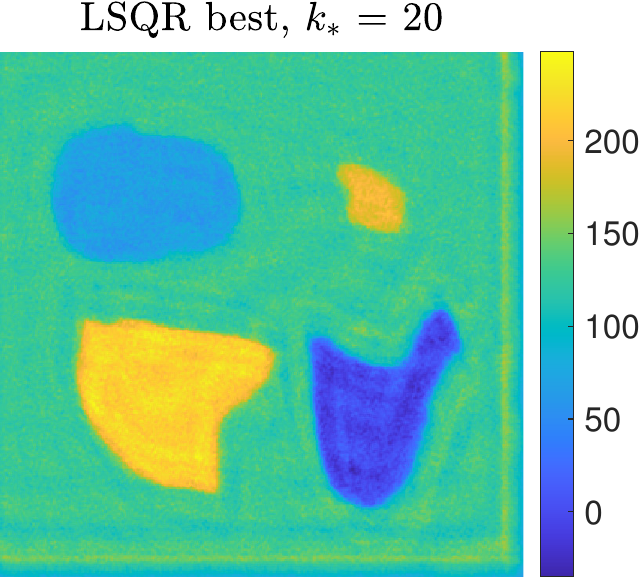}}\hspace{-7mm}
\subfigure{
\includegraphics[width=0.225\textwidth]{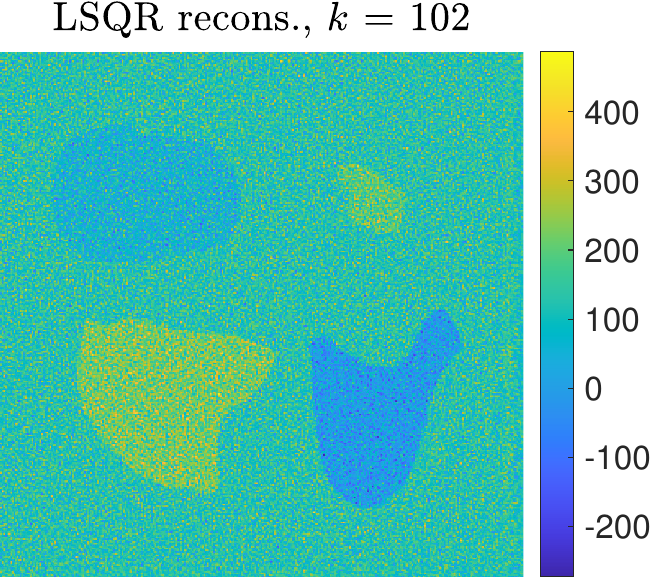}}\hspace{-7.4mm}
\subfigure{
\includegraphics[width=0.22\textwidth]{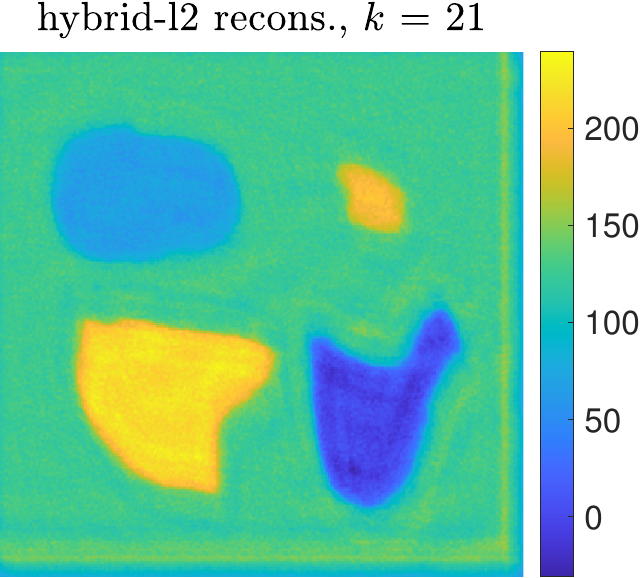}}
\vspace{-3mm}
\caption{The reconstructed images computed by iDARR, LSQR and hybrid-l2 methods.} 
\label{fig:deblur}
\end{figure}
\fi 

\ifarXiv

\begin{figure}[htb]
\centering
\subfigure{
\includegraphics[width=0.22\textwidth]{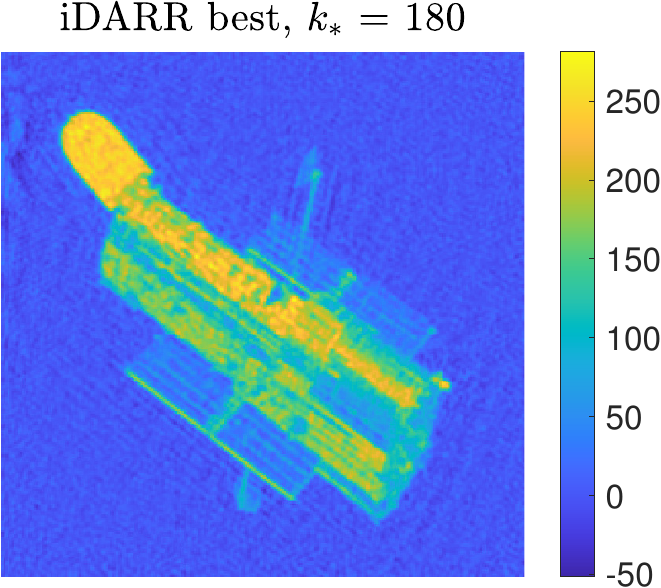}}\hspace{-9.2mm}
\subfigure{
\includegraphics[width=0.22\textwidth]{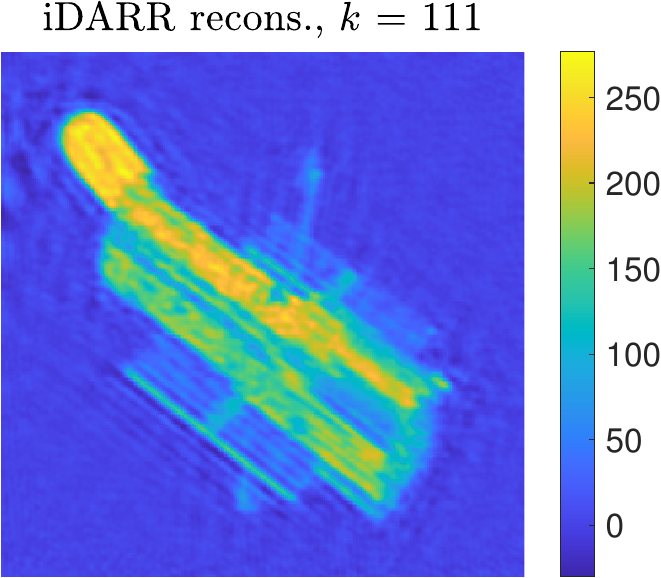}}\hspace{-9.2mm}
\subfigure{
\includegraphics[width=0.22\textwidth]{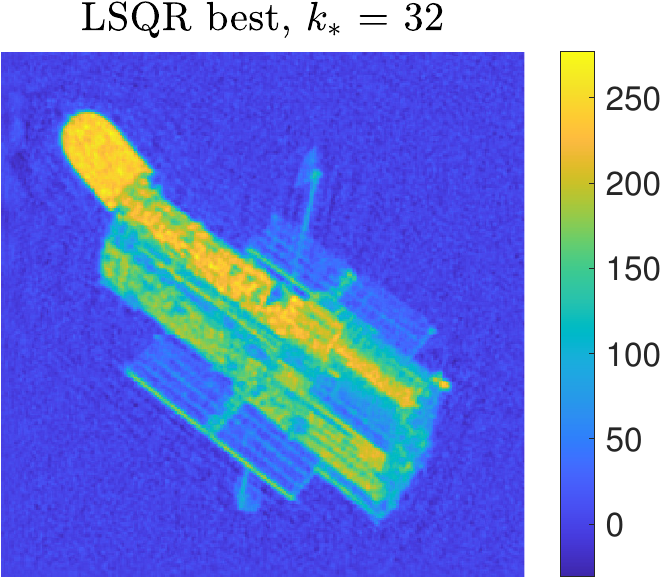}}\hspace{-9.2mm}
\subfigure{
\includegraphics[width=0.22\textwidth]{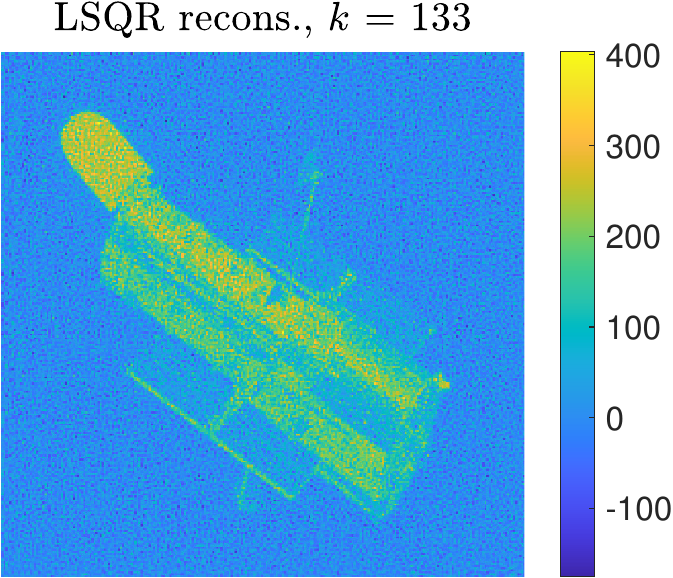}}\hspace{-9.4mm}
\subfigure{
\includegraphics[width=0.22\textwidth]{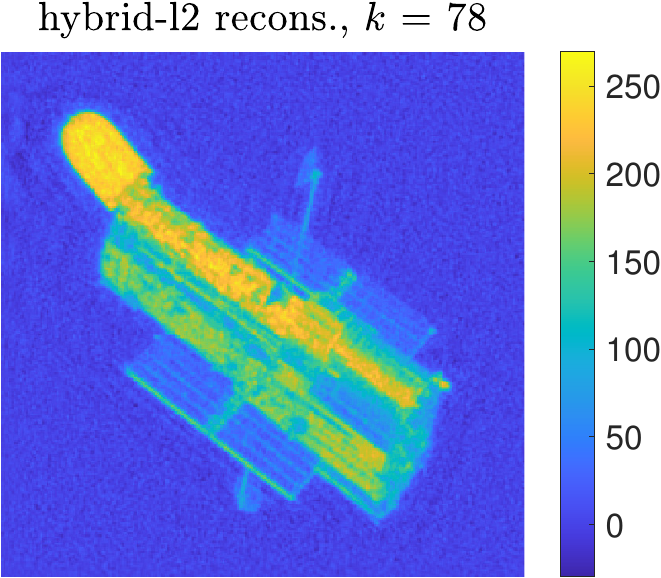}} \\
\vspace{-2mm}
\subfigure{
\includegraphics[width=0.22\textwidth]{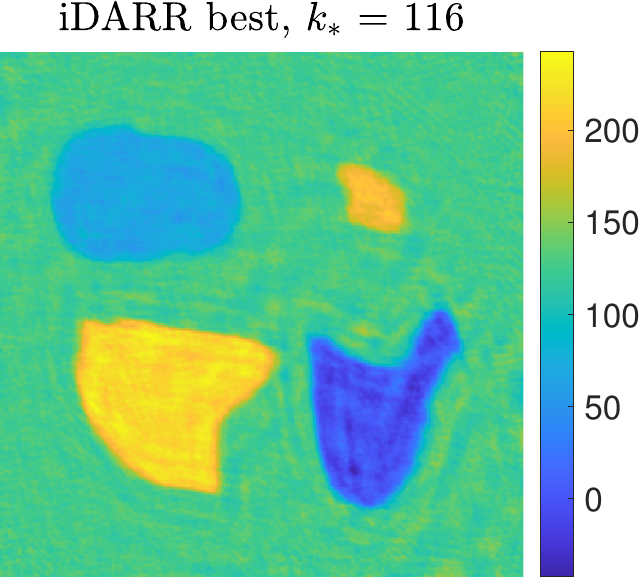}}\hspace{-8.7mm}
\subfigure{
\includegraphics[width=0.22\textwidth]{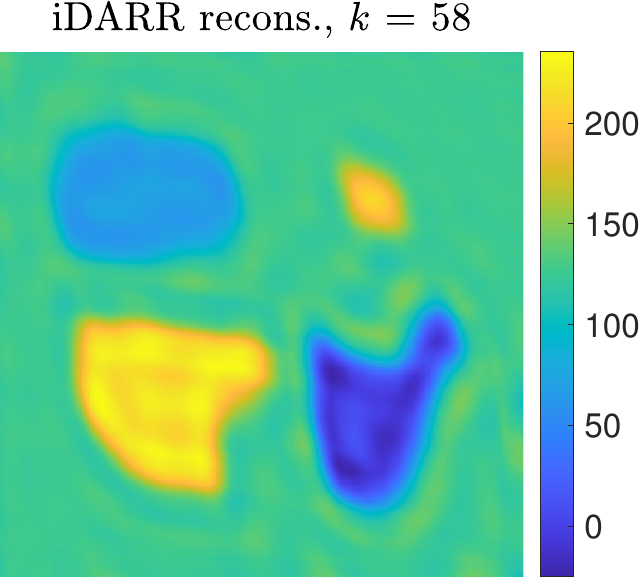}}\hspace{-8.7mm}
\subfigure{
\includegraphics[width=0.22\textwidth]{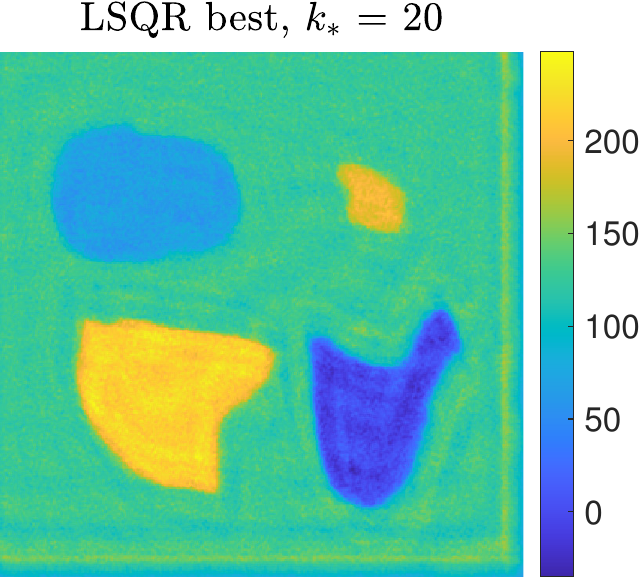}}\hspace{-8.7mm}
\subfigure{
\includegraphics[width=0.22\textwidth]{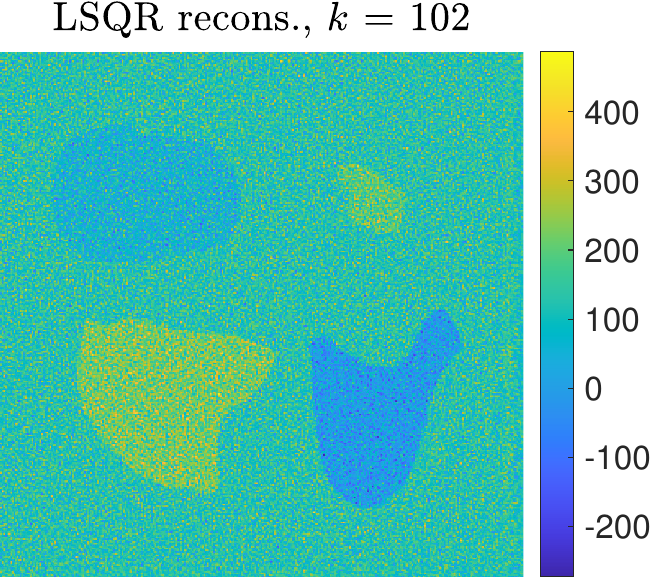}}\hspace{-9.2mm}
\subfigure{
\includegraphics[width=0.22\textwidth]{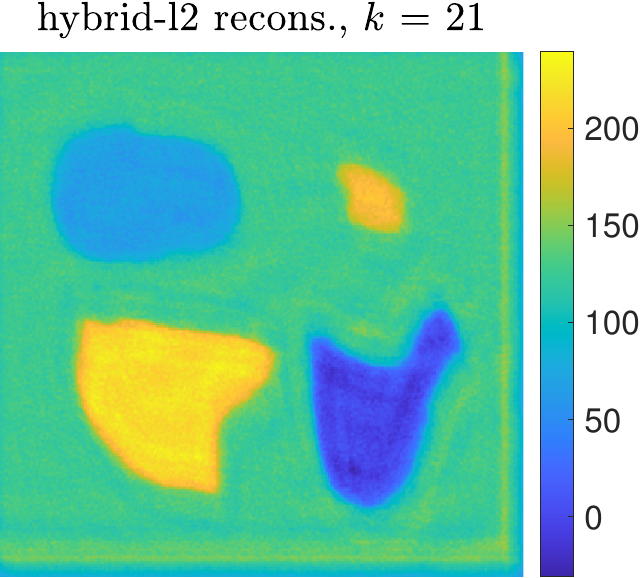}}
\vspace{-3mm}
\caption{The reconstructed images computed by iDARR, LSQR and hybrid-l2 methods.} 
\label{fig:deblur}
\end{figure}
\fi

\begin{figure}[!htb] \vspace{-3mm}
\centering
\subfigure[Relative error, Image-1]{
\includegraphics[width=0.3\textwidth]{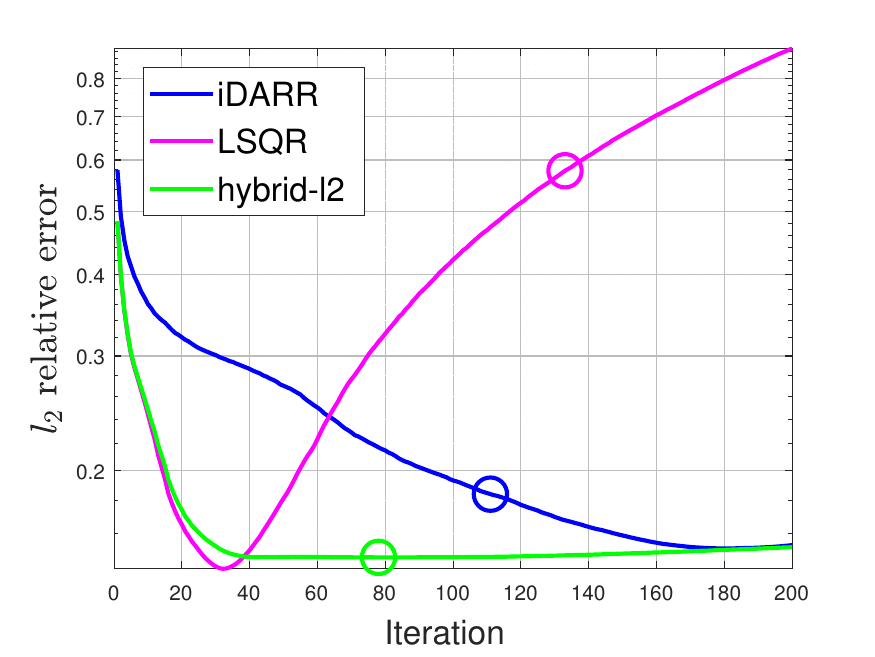}}\hspace{-4mm}
\subfigure[iDARR, Image-1]{
\includegraphics[width=0.3\textwidth]{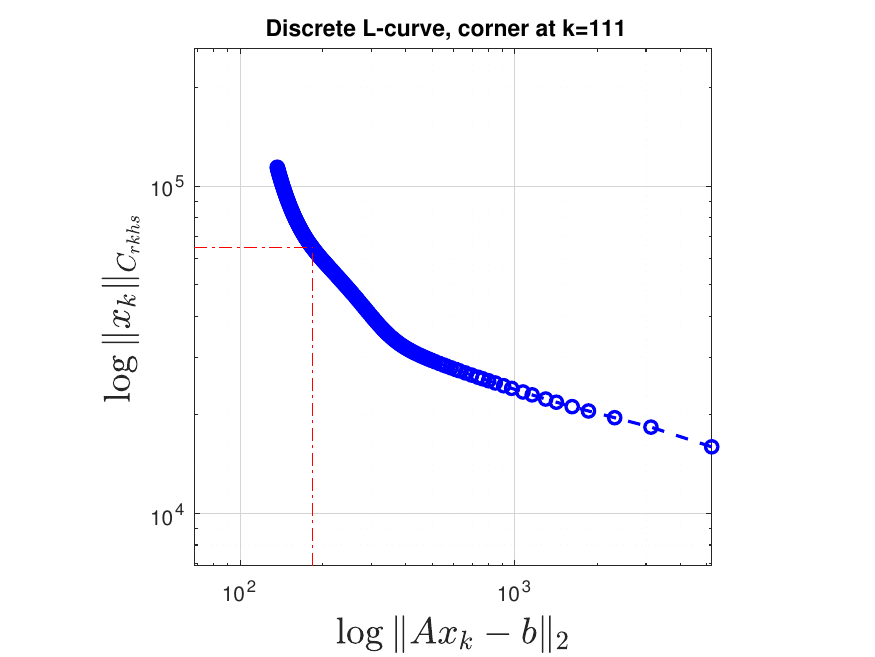}}\hspace{-5mm}
\subfigure[LSQR, Image-1]{
\includegraphics[width=0.3\textwidth]{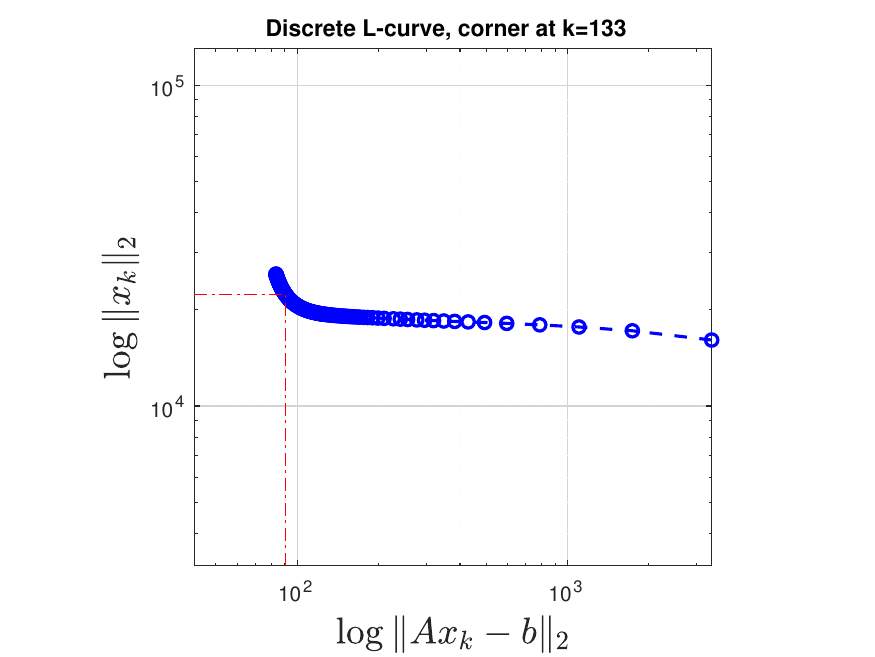}} \\
\vspace{-3mm}
\subfigure[Relative error, Image-2]{
\includegraphics[width=0.3\textwidth]{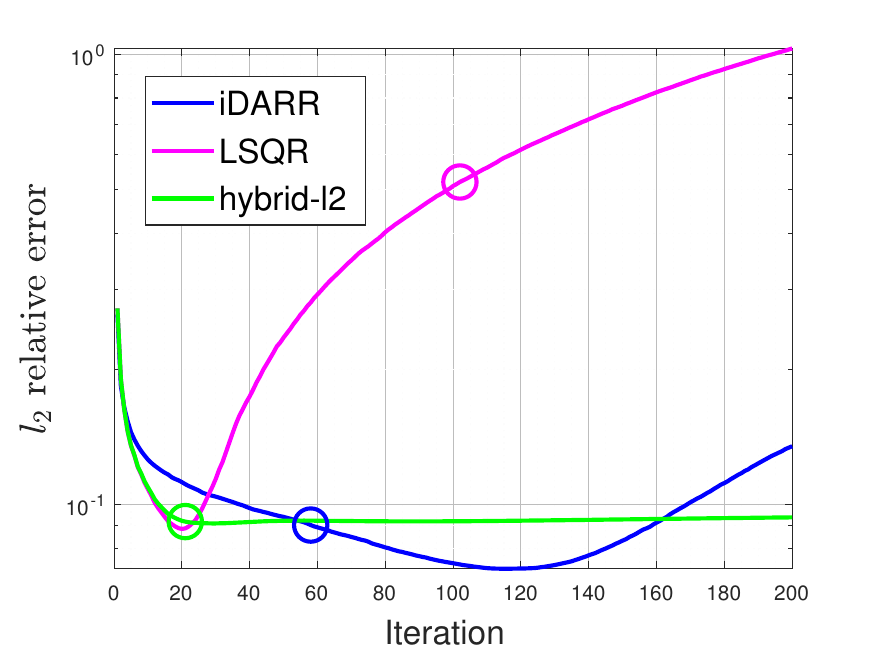}}\hspace{-4mm}
\subfigure[iDARR, Image-2]{
\includegraphics[width=0.3\textwidth]{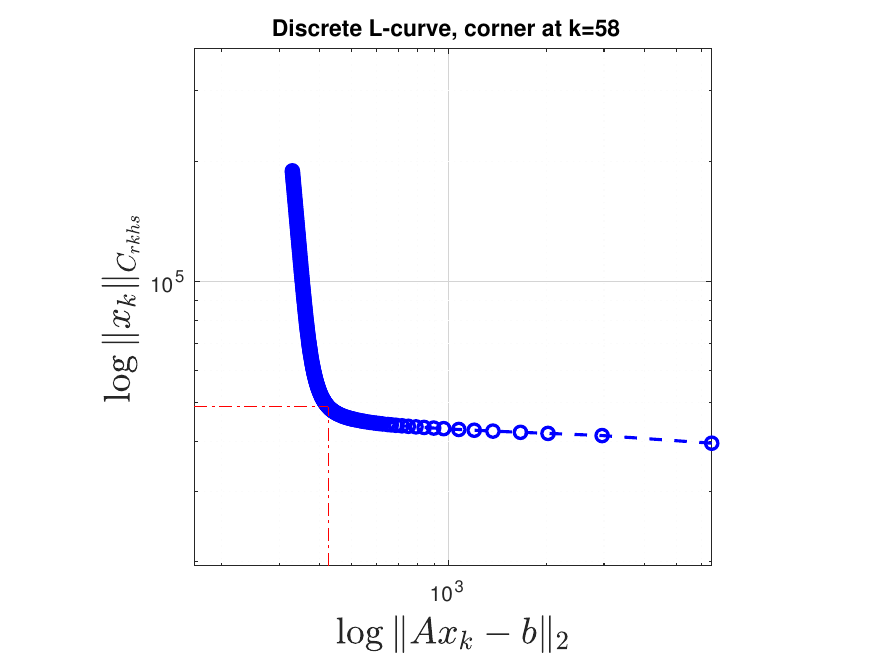}}\hspace{-5mm}
\subfigure[LSQR, Image-2]{
\includegraphics[width=0.3\textwidth]{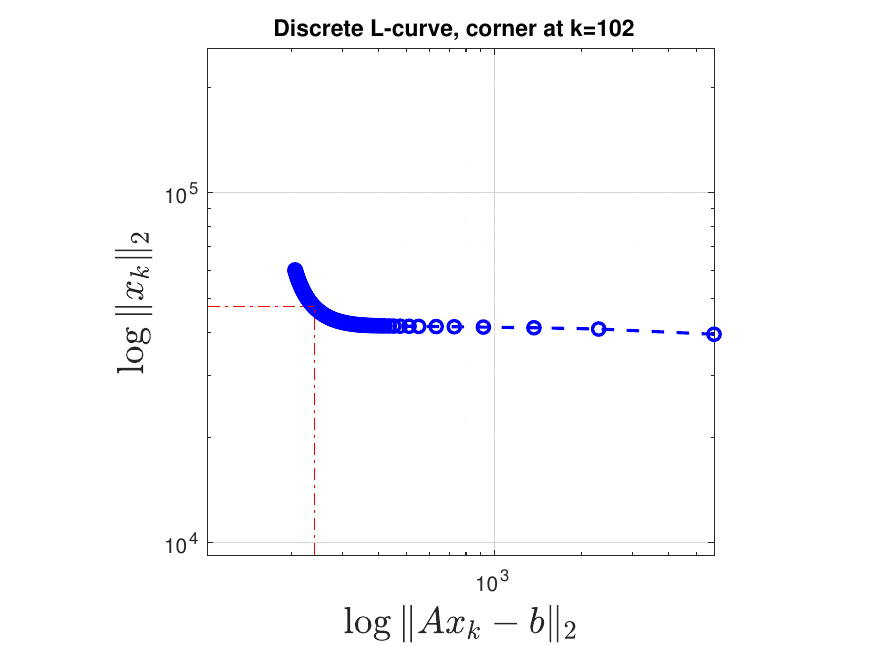}}
\vspace{-3mm}
\caption{Relative errors in iteration numbers, where the circles mark the early stopping iterations chosen by the L-curve method presented in the right two columns. 
}
\vspace{-4mm}
\label{fig:rel_error}
\end{figure}

\Cref{fig:deblur} shows the reconstructed images computed by iDARR, LSQR, and the hybrid-l2 methods. Here the hybrid-l2 applies an $l_2$-norm Tikhonov regularization to the projected problem obtained by LSQR, and it uses the stopping strategy in \cite{gazzola2019ir}. The best estimations of for iDARR or LSQR are solutions with $k_*$ minimizing $\|x_{k}-x_{\text{true}}\|_2$. Their reconstructed solutions are obtained by using the L-curve method for early stopping.  

\Cref{fig:rel_error} (a)--(f) show the relative errors as the iteration number increases and the selection of the early stopping iterations by the L-curve method in the right two columns.

Notably, \Cref{fig:rel_error} reveals that iDARR achieves more accurate reconstructed images than LSQR for both tests, despite appearing to the contrary in Figure {fig:deblur}. The LSQR appears prone to stopping late, resulting in lower-quality reconstructions than iDARR. In contrast, iDARR tends to stop earlier than ideal, before achieving the best quality. However, the hybrid-l2 method consistently produces accurate estimators with stable convergence, suggesting potential benefits in developing a hybrid iDARR approach to enhance stability. 

The effectiveness of iDARR depends on the alignment of regularities between the convolution kernel and the image, as the DA-RKHS's regularity is tied to the smoothness of the convolution kernel. With the {\sf PRblurspeckle} featuring a smooth PSF, iDARR obtains a higher accuracy for the smoother Image-2 compared to Image-1, producing reconstructions with smooth edges. An avenue for future exploration involves adjusting the DA-RKHS's smoothness to better align with the smoothness of the data.

%%%%%%%%%%%%%%%-------------------------------------------------------------

\section{Conclusion and Future Work}\label{sec:conlusion}
We have introduced iDARR, a scalable iterative data-adaptive RKHS regularization method, for solving ill-posed linear inverse problems. It searches for solutions in the subspaces where the true signal can be identified and achieves reliable early stopping via the DA-RKHS norm. A core innovation is a generalized Golub-Kahan bidiagonalization procedure that recursively computes orthonormal bases for a sequence of RKHS-restricted Krylov subspaces. Systematic numerical tests on the Fredholm integral equation show that iDARR outperforms the widely used iterative regularizations using the $L^2$ and $l^2$ norms, in the sense that it produces stable accurate solutions consistently converging when the noise level decays. Applications to 2D image de-blurring further show the iDARR outperforms the benchmark of LSQR with the $l^2$ norm.

\paragraph{Future Work: Hybrid Methods}
The accuracy and stability of the regularized solution hinges on the choice of iteration number for early stopping. While the L-curve criterion is a commonly used tool for determining this number, it can sometimes lead to suboptimal results due to its reliance on identifying a corner in the discrete curve. Hybrid methods are well-recognized alternatives that help stabilize this semi-convergence issue, as referenced in \cite{Kilmer2001,Chungnagy2008,Renaut2017}. One promising approach is to apply Tikhonov regularization to each iteration of the projected problem. The hyperparameter for this process can be determined using the \emph{weighted generalized cross-validation method} (WGCV) as described in \cite{Chungnagy2008}. This approach is a focus of our upcoming research project.

 \ifjournal   \bibliographystyle{siamplain} \fi
 \ifarXiv \bibliographystyle{plain} \fi
\bibliography{ref_invLap,ref_regularization23_09,ref_FeiLu2023_10,Iter_Regu}
\end{document}